\documentclass[a4paper]{amsart}
\usepackage[T1]{fontenc}
\usepackage[utf8]{inputenc}
\usepackage{lmodern}
\usepackage{enumerate}
\usepackage{amssymb,amsxtra}
\usepackage[all]{xy}
\usepackage{xcolor}
\usepackage{nicefrac,mathtools}
\usepackage{microtype}
\usepackage{amscd}
\usepackage{geometry}
\usepackage{mathbbol}
\usepackage{comment}

\usepackage{comment}

\usepackage[pdftitle={...},
 pdfauthor={Dami\'an Ferraro},
 pdfsubject={Mathematics}]{hyperref}
\usepackage{cite}

\numberwithin{equation}{section}
\theoremstyle{plain}
\newtheorem{theorem}[equation]{Theorem}
\newtheorem{lemma}[equation]{Lemma}
\newtheorem{proposition}[equation]{Proposition}
\newtheorem{corollary}[equation]{Corollary}
\theoremstyle{definition}
\newtheorem{definition}[equation]{Definition}
\newtheorem{notation}[equation]{Notation}
\theoremstyle{remark}
\newtheorem{remark}[equation]{Remark}

\newtheorem{alphthm}{Theorem}			

\newcommand*{\C}{\mathbb C}

\newcommand{\CLB}{\mathrm{CLB}}

\newcommand{\Ad}{\operatorname{Ad}}

\newcommand{\bB}{\mathbb{B}}
\newcommand{\bC}{\mathbb{C}}

\newcommand{\bk}{\mathbb{k}}
\newcommand{\bo}{\bar{\otimes}}

\newcommand{\bK}{\mathbb{K}}
\newcommand{\bL}{\mathbb{L}}
\newcommand{\bM}{\mathbb{M}}
\newcommand{\bN}{\mathbb{N}}

\newcommand{\cA}{\mathcal{A}}
\newcommand{\cB}{\mathcal{B}}
\newcommand{\cC}{\mathcal{C}}

\newcommand{\cL}{\mathcal{L}}

\newcommand{\cX}{\mathcal{X}}

\newcommand{\fqb}[2]{\dot{#1}_{#2}}
\newcommand{\fqbr}[2]{\dot{#1}^\red_{#2}}

\newcommand{\id}{\operatorname{id}}
\newcommand{\indfun}{\mathbb{1}}
\newcommand{\Ind}{\operatorname{Ind}}
\newcommand{\intInd}{\widetilde{\operatorname{Ind}}}

\newcommand{\li}[1]{L^\infty(#1)}
\newcommand{\lone}[1]{L^1(#1)}
\newcommand{\ltwo}[1]{L^2(#1)}

\newcommand{\otmax}{\otimes_{\max}}
\newcommand{\otmin}{{\otimes}}

\newcommand{\qb}[2]{{#1}_{/#2}}
\newcommand{\qbr}[2]{{#1}^\red_{/#2}}

\newcommand{\smu}{{s^{-1}}}

\newcommand{\sot}{\texttt{sot}}
\newcommand{\tmu}{{t^{-1}}}

\newcommand{\wcp}{\texttt{wcp}}
\newcommand{\wstar}{{\operatorname{\rm w}^*}}
\newcommand{\wot}{\texttt{wot}}

\newcommand{\M}{\mathcal M} 
\newcommand{\sbe}{\subseteq}

\providecommand{\cspn}{\overline{\mathop{\rm span}}}
\providecommand{\Ind}{{\mathop{\rm Ind}}}
\providecommand{\red}{{\mathop{\rm r}}}
\providecommand{\spn}{{\mathop{\rm span}}}
\providecommand{\supp}{{\mathop{\rm supp}}}

\newcommand{\Cstar}{\texorpdfstring{$\mathrm C^*$}{C*}}
\newcommand{\Wstar}{\texorpdfstring{$\mathrm W^*$}{W*}}
\newcommand*{\Star}{$^*$\nobreakdash-\hspace{0pt}}

\DeclarePairedDelimiterX{\braket}[2]{\langle}{\rangle}{#1\,\delimsize\vert\,\mathopen{}#2}

\title[Amenability for noncommutative dynamical systems and Fell bundles]{Characterizations of amenability for noncommutative dynamical systems and Fell bundles}
\author{Alcides Buss}
\author{Damián Ferraro}

\dedicatory{To the memory of Fernando Abadie, 1964 - 2024.}



\begin{document}

\begin{abstract}
We resolve key open questions regarding approximation properties and their permanence for Fell bundles over locally compact groups. Specifically, we establish the equivalence between the B\'edos--Conti approximation property (BCAP) and the Exel--Ng positive approximation property (AP), completely removing the necessity of assuming nuclearity on the unit fiber.

To overcome the obstructions present in general Fell bundles (such as the lack of spatial arguments and exactness), we introduce a diagonal maximal tensor product $\otimes^d_{\max}$. We prove that a Fell bundle $\mathcal{A}$ has the AP if and only if $\mathcal{A} \otimes^d_{\max} \mathcal{B}$ has the weak containment property (wcp) for every Fell bundle $\mathcal{B}$. For $C^*$-dynamical systems, this yields a characterization of amenability that was known to hold under exactness assumptions.

Furthermore, this tensorial machinery allows us to establish highly non-trivial permanence properties for the AP, including passage to restrictions over closed subgroups and partial quotients by normal subgroups. We also provide applications concerning the nuclearity of full and reduced cross-sectional $C^*$-algebras.
\end{abstract}

\maketitle

\section{Introduction}
Amenability is a central notion connecting group theory, operator algebras and noncommutative dynamics.
In the setting of actions of locally compact groups on von Neumann algebras, amenability originates in Zimmer's work on measurable dynamics and was studied systematically by Anantharaman-Delaroche; see\ \cite{ADaction1979,ADactionII1982}. In this context, amenability is often referred to as \emph{Zimmer amenability}.

For actions on \Cstar-algebras, Anantharaman-Delaroche introduced an amenability notion for discrete groups via the bidual \cite{ADsystemes1987}. However, for a general locally compact group $G$ this approach does not extend directly, since the induced action on the bidual $A^{**}$ of a $G$-\Cstar-algebra $A$ need not be $w^*$-continuous. This conceptual difficulty was resolved more recently by Buss-Echterhoff-Willett \cite{BssEff_amenability} and, independently, by Ozawa-Suzuki \cite{ozawa2021characterizations}, building on Ikunishi’s theory \cite{Ikunishi88} of the enveloping $G$-\Wstar-algebra $A''_\alpha$. These works show that several natural formulations of amenability for \Cstar-dynamical systems of locally compact groups -- including definitions via the enveloping \Wstar-algebra and approximation properties -- are equivalent.

Since these foundational developments, a rich theory has emerged relating amenability to approximation properties, tensorial behaviour, and structural features of crossed products. Throughout the paper, the term \emph{amenable} for \Cstar-actions is used in this modern sense, and we freely rely on the equivalent characterizations established in the works cited above.

Despite this progress, extending these characterizations and permanence properties from group actions to general Fell bundles presents severe technical obstructions. In the classical setting of group actions, permanence properties of amenability (such as restriction to subgroups or quotients) often rely on spatial arguments, exactness of the group, or nuclearity. For general Fell bundles, these techniques break down entirely, particularly because the unit fiber $B_e$ need not be nuclear and exactness of the base group is not guaranteed. 

The primary goal of this paper is to resolve these open questions regarding approximation properties and permanence in Fell bundles. Most notably, we establish the equivalence of B\'edos--Conti's approximation property (BCAP) and Exel--Ng's positive approximation property (AP) without any nuclearity hypotheses, and we prove the passage of amenability to subgroups and quotients.

To bypass the classical obstructions, our strategy relies on introducing a new tensorial mechanism. First, we give a tensorial (``diagonal'') criterion for amenability of actions: amenability of a system can be tested by forming diagonal maximal tensor products with arbitrary test-systems and checking the weak containment property (wcp). Second, we extend this viewpoint to Fell bundles by constructing a natural diagonal maximal tensor product $\otimes_{\max}^d$ and showing that the AP is equivalent to a diagonal tensorial wcp characterization.

Below we state our main contributions in a compact, theorem-style form so the reader can immediately see the scope of the paper. After the statements we provide comments that highlight important subtleties and limitations (in particular concerning reductions to subgroups and quotient bundles) which are treated carefully in Section~\ref{sec: reductions and quotients}.

\begin{alphthm}\label{A: amenability-actions}
Let $G$ be a locally compact group and let $\alpha$ be a \Cstar-action of $G$ on a \Cstar-algebra $A$. Using the universal unitary implementation $(A\subset \bB(X_\alpha),U^{\alpha})$, denote by $A'_\alpha$ the commutant of $A$ in $\bB(X_\alpha)$ and let $\beta$ be the continuous part of $\Ad(U^{\alpha})|_{A'_\alpha}$. The following are equivalent:
\begin{enumerate}
  \item $\alpha$ is amenable (i.e., the enveloping \Wstar-action $\alpha''$ on $A''_\alpha$ is \Wstar-amenable in the sense of Anantharaman-Delaroche).
  \item For every \Cstar-action $\gamma$ of $G$, the diagonal maximal tensor action $\alpha\otimes_{\max}^d\gamma$ has the weak containment property (wcp), i.e.\ the canonical quotient
  \[
  (A\otimes_{\max} B)\rtimes_{\alpha\otimes^d_{\max}\gamma} G \twoheadrightarrow (A\otimes_{\max} B)\rtimes_{\red\, (\alpha\otimes^d_{\max}\gamma)} G
  \]
  is an isomorphism.
  \item The diagonal action $\alpha\otimes_{\max}^d\beta$ has the wcp for the specific canonical test action $\beta$ defined above.
\end{enumerate}
\end{alphthm}

Theorem~\ref{A: amenability-actions} gives a practical tensorial test for amenability: (2) is a full tensorial characterization while (3) shows that a single, canonical test-action suffices. Our proof removes the need to assume exactness of $G$ (previously required in some formulations) by exploiting a careful analysis of the universal unitary implementation and a von Neumann tensoring technique for equivariant completely positive maps (see Appendix~\ref{sec: appendix tensor product of cp equivariant maps}).

\begin{alphthm}\label{B: fell-bundles-diagonal}
Let $\cA$ be a Fell bundle over a locally compact group $G$. There is a natural diagonal maximal tensor product $\otimes_{\max}^d$ for Fell bundles over the same group. With this notation, the following are equivalent:
\begin{enumerate}
  \item $\cA$ has Exel--Ng's positive approximation property (AP).
  \item For every Fell bundle $\cB$ over $G$, the diagonal Fell bundle $\cA\otimes_{\max}^d\cB$ has the weak containment property (i.e.\ its full and reduced cross-sectional \Cstar-algebras coincide).
  \item The canonical action of $G$ on the \Cstar-algebra of kernels of $\cA$ (as in \cite{Ab03}) is amenable in the sense of Anantharaman-Delaroche.
\end{enumerate}
\end{alphthm}

Theorem~\ref{B: fell-bundles-diagonal} unifies Exel--Ng's AP with a dynamical notion of amenability for the kernel action and extends the diagonal tensor viewpoint from actions to general Fell bundles.

\begin{alphthm}\label{C: BCAP-permanence-nuclearity}
The following permanence and nuclearity statements hold for actions and Fell bundles over a locally compact group $G$:
\begin{enumerate}
  \item (BCAP equivalence) The approximation property introduced by B\'edos--Conti (BCAP), suitably extended to Fell bundles over locally compact groups, is equivalent to Exel--Ng's AP (no nuclearity hypothesis on the unit fiber is required). This establishes BCAP as a conceptually and technically robust characterization of amenability for Fell bundles.
  \item (Restriction to subgroups) If $\cA$ has the AP and $H\le G$ is a closed subgroup, then the restriction $\cA|_H$ has the AP.
  \item (Normal subgroup quotients) If $H\trianglelefteq G$ is a normal closed subgroup and $\qb{\cB}{H}$ denotes the quotient Fell bundle over $G/H$ constructed from $\cA$ (see remarks below), then $\qb{\cB}{H}$ has the AP whenever $\cA$ has the AP.
  \item (Nuclearity interplay) A Theorem of F. Abadie states that if $\cA$ has the AP and $A_e$ is nuclear, then $C^*(\cA)=C^*_\red(\cA)$ is nuclear.
  Conversely, if $C^*_\red(\cA)$ is nuclear and $G$ inner amenable, then $A_e$ is nuclear and $\cA$ has the AP (we give self contained proofs for SIN groups).
\end{enumerate}
\end{alphthm}

Item (1) of Theorem~\ref{C: BCAP-permanence-nuclearity} deserves special emphasis: B\'edos and Conti introduced BCAP in the discrete group setting and proved that AP implies BCAP (and the converse under nuclearity of the unit fiber). Our contribution is to extend BCAP to the locally compact setting and to show, without imposing nuclearity on the unit fiber, that BCAP and AP are equivalent. This equivalence both clarifies the landscape of approximation properties and provides a powerful, alternate toolkit to detect amenability in examples.

\subsection*{Remarks on reductions and quotient bundles}
A recurring theme in Section~\ref{sec: reductions and quotients} is the relationship between \Cstar-amenability (i.e. AP) of a Fell bundle $\cB$ over $G$ and the \Cstar-amenability of its restriction $\cB_H$ to a closed subgroup $H$ and of the bundle obtained by forming the (partial) quotient over $G/H$. The results in Section~\ref{sec: reductions and quotients} show the following facts, proved there in full detail:
\begin{itemize}
  \item If $\cB$ is \Cstar-amenable then $\cB_H$ is \Cstar-amenable for every closed subgroup $H$ (Theorem~\ref{thm: amenability and reductions}).
  \item When $H$ is open (or more generally, when its normalizer in $G$ is open) we use results from \cite{Ferraro_2024} about the existence of conditional expectations $C^*(\cB)\to C^*(\cB_H)$ and $C^*_\red(\cB)\to C^*_\red(\cB_H)$, so nuclearity and the \wcp\ pass from $\cB$ to $\cB_H$ in that case.
  \item If $H$ is normal we construct full and reduced completions of the partial cross-sectional bundle over $G/H$, denoted $\qb{\cB}{H}$ and $\qbr{\cB}{H}$ respectively, and prove isomorphisms $C^*(\cB)\cong C^*(\qb{\cB}{H})$ and $C^*_\red(\cB)\cong C^*_\red(\qbr{\cB}{H})$; moreover, under the hypothesis that $\cB_H$ has the \wcp\ the two completions coincide.
\end{itemize}

These constructions allow us to deduce many permanence results, but they also uncover a nuanced limitation: while the implication
\[
\text{$\cB$ is \Cstar-amenable} \quad\Longrightarrow\quad \text{$\cB_H$ and }\qb{\cB}{H}\ \text{are \Cstar-amenable}
\]
holds (in fact the forward direction is proved generally in the paper), the converse
\[
\text{$\cB_H$ and }\qb{\cB}{H}\ \text{are \Cstar-amenable} \quad\Longrightarrow\quad \text{$\cB$ is \Cstar-amenable}
\]
is not established in full generality in this article. We prove this converse under additional hypotheses (for instance, when $G$ is inner amenable and the unit fibre $B_e$ is nuclear; see Corollary~\ref{cor: completion nuclear}), and we know the converse holds for discrete $G$ (this will be treated in detail in a forthcoming paper). Thus, Section~\ref{sec: reductions and quotients} contains both positive permanence results and a careful account of what is currently known about converses and the hypotheses required for them.

\subsection*{Methods and structure of the paper}

The proofs combine three main ingredients:
\begin{itemize}
  \item an analysis of the universal unitary implementation of an action and its commutant (used to reduce amenability questions to equivariant conditional expectations);
  \item a construction of diagonal maximal tensor products for Fell bundles based on F. Abadie's maximal tensor product of bundles; and
  \item a technical toolkit to form tensor products of equivariant completely positive maps between von Neumann algebras (developed in Appendix~\ref{sec: appendix tensor product of cp equivariant maps}).
\end{itemize}

The paper is organized as follows. Section~\ref{sec: amenability of dynamical systems} contains preliminaries on \Cstar- and \Wstar-actions and the proof of Theorem~\ref{A: amenability-actions}. Section~\ref{sec: amenability of Fell bundles} constructs $\otimes_{\max}^d$ for Fell bundles and proves Theorem~\ref{B: fell-bundles-diagonal}. Section~\ref{sec: the BCAP} establishes equivalence of BCAP and AP and collects permanence results, including those listed in Theorem~\ref{C: BCAP-permanence-nuclearity}. Section~\ref{sec: reductions and quotients} investigates restrictions, quotients and other permanence properties; here we give the most complete proofs of the forward implications and indicate precisely the extra hypotheses needed for the known converses. Two appendices contain auxiliary technical lemmas used throughout the paper.

\subsection*{Acknowledgements}
We dedicate this work to the memory of Fernando Abadie. The second author warmly thanks the Departamento de Ci\^{e}ncias F\'isicas e Matem\'aticas at Universidade Federal de Santa Catarina for hospitality and support. The first author was supported by CNPq and FAPESC.

\section{Amenability of dynamical systems}\label{sec: amenability of dynamical systems}

\subsection{Preliminaries}\label{ssec:preliminaries}

We start by presenting our notation and some basic facts.
Some results are stated and proved in the Appendix.

In this article, we use \textit{group} as an abbreviation of \textit{locally compact and Hausdorff group}.
A unitary representation $U$ of a group $G$ on (a Hilbert space) $X$ is a group homomorphism $t\mapsto U_t$ from $G$ to the unitary operators of $X$ such that, for all $x,y\in X,$ $(t\in G)\mapsto \langle U_tx,y\rangle$ is continuous. We convention that all of our inner products are linear in the second variable.

For us, a von Neumann algebra (or \Wstar-algebra) is a \Cstar-algebra that is linearly and isometrically isomorphic to the dual of a Banach space. Faithful normal representations need not be nondegenerate (i.e. unital), and even so we use them to construct faithful representations of the von Neumann tensor product $\bo$. By definition, \Wstar\textit{-subalgebra} means a (unital) $\wstar$-closed \Star{}subalgebra. The minimal and maximal tensor products of \Cstar-algebras will be denoted $\otmin$ and $\otmax,$ respectively.
The multiplier algebra and the center of a \Cstar-algebra $A$ will be denoted $\M(A)$ and $Z(A)$, respectively.

When we say that $\gamma$ is a \Wstar-action (respectively, a \Cstar-action) of a group $G$ on a \Wstar-algebra (\Cstar-algebra) $M$, we mean that $\gamma \colon G \times M \to M$, $(t,m) \mapsto \gamma_t(m)$, is a set-theoretic action such that each $\gamma_t$ is a \Star{}isomorphism and, for all $m \in M$, the map $t \mapsto \gamma_t(m)$ is $\wstar$-continuous (norm continuous). A \emph{covariant representation} of $\gamma$ is a triple $(X, \pi, U)$ where $\pi \colon M \to \bB(X)$ is a \Star{}representation, $U$ is a unitary representation of $G$ on $X$, and for all $t \in G$ and $m \in M$, we have $\pi(\gamma_t(m)) = U_t \pi(m) U_t^* =: \Ad(U)_t(\pi(m))$. If we qualify a covariant representation with a property (e.g., normal, unital, nondegenerate), we mean that the representation $\pi$ has that property. A \emph{unitary implementation} of the \Wstar-action (respectively, \Cstar-action) $\gamma$ of $G$ on $M$ is a covariant representation $(X,\pi,U)$ of $\gamma$ that is faithful, unital, and normal (respectively, faithful and nondegenerate). In this case, we say that $(M \sbe \bB(X), U)$ is a unitary implementation of $\gamma$, and we may omit explicit reference to $\pi$ when no confusion can arise. A triple $(G, A, \alpha)$ is called a \Cstar- (respectively, \Wstar-) dynamical system if $\alpha$ is a \Cstar-(\Wstar-)action of $G$ on $A$.

As in \cite{ADaction1979}, we say that a \Wstar-action $\gamma$ of $G$ on $M$ is amenable (or \Wstar-amenable) if there exists an equivariant conditional expectation $P\colon \li{G} \bo M\to M$, where we identify each $m\in M$ with $1\otimes m\in \li{G}\bo M $ and the action of $G$ on $\li{G} \bo M$ is the diagonal action induced by conjugation by the left regular representation $\lambda\colon G\to \bB(\ltwo{G})$ and $\gamma$.
It was first noticed in \cite{ADaction1979,ADactionII1982} that the amenability of $\gamma$ is related to that of the restriction of $\gamma$ to the center of $M$, $\gamma|_{Z(M)}$.
Anantharaman-Delaroche proved they are equivalent under separability conditions and the general case is a consequence of \cite[Theorem 3.17]{BssEff_amenability}.

The continuous part of $\gamma$, denoted $\gamma^c$, is the restriction of $\gamma$ to the \Cstar-subalgebra $M^c\sbe M$ formed by those $m\in M$ such that $t\mapsto \gamma_t(m)$ is norm continuous.
Recall from \cite{ADactionII1982} that $M^c$ is $\wstar$-dense in $M$ and notice that $\gamma^c$ is a \Cstar-action.
The separability assumptions made in the introduction of \cite{ADactionII1982} are not necessary in the proof of \cite[Lemme 2.1]{ADactionII1982}, so $\gamma$ is \Wstar-amenable if and only if there exists an equivariant conditional expectation $P\colon (\li{G}\bo Z(M))^c\to Z(M)^c$. We used Anantharan-Delaroche's ideas to prove the extension lemma of Appendix~\ref{sec: an extension lemma}.

In Appendix~\ref{sec: appendix tensor product of cp equivariant maps}, we address the problem of constructing the (von Neumann) tensor product of possibly non-normal, equivariant, contractive completely positive (ccp) maps between \Wstar-algebras. Our approach combines techniques from \cite[Lemme 2.1]{ADactionII1982}, \cite[Theorem 3]{NagTomCP}, and \cite[Theorem 4]{Tomiyama_tensor_vN_algebras}.

The notion of amenability of a \Cstar-action $\alpha$ of $G$ on $A$ requires the use of the enveloping $G$-\Wstar-algebra of $\alpha,$ $A''_\alpha$, which we introduce here using the presentation of \cite{BssEff_amenability}.
Let $A\rtimes_\alpha G$ be the (full) crossed-product of $\alpha$ and $A\rtimes_\alpha G\sbe \bB(X_\alpha)$ the universal representation.
This representation is the integrated form of the ``universal'' unitary implementation $(A\sbe \bB(X_\alpha),U^\alpha)$.
By definition, $A''_\alpha$ is the bicommutant of $A$ in $\bB(X_\alpha)$ and the \Wstar-action $\alpha''$ of $G$ on $A''_\alpha$ is given by $(t,m)\mapsto U^\alpha_t m U^\alpha_\tmu\equiv \Ad(U^\alpha)_t (m).$
We say that $\alpha$ is \Cstar-amenable if $\alpha''$ is \Wstar-amenable.

We recall here some facts we will need later on.
Given a covariant representation $(X,\pi,U)$ of $\alpha$, when no confusion can arise, we write $A\rtimes G$ instead of $A\rtimes_\alpha G$.
In general, $\pi\rtimes U\colon A\rtimes_\alpha G\to \bB(X)$ will denote the integrated form of $(\pi,U)$.
For all $t\in G$ and $f\in C_c(G,A)$ we have $U_t\pi\rtimes U(f)= \pi\rtimes U(\tilde{\alpha}_t(f))$ with $\tilde{\alpha}_t(f)(s)=\alpha_t(f(\tmu s)).$
Similarly, for $a\in A$ we define $af$ as the pointwise product of $a$ and $f$ and we have $\pi(a)\pi\rtimes U(f)=\pi\rtimes U(af)$.

The reduced covariant representation $(\ltwo{G}\otimes X,1\pi,\lambda U)$ associated to $(X,\pi,U)$ is given by $1\pi(a)=1\otimes \pi(a)$ and $\lambda U_t= \lambda_t\otimes U_t$; where $\lambda\colon G\to \bB(\ltwo{G})$ is the left regular representation.
Recall that $1\pi\rtimes \lambda U$ factors through a representation $1\pi\rtimes_\red \lambda U$ of the reduced crossed product $A\rtimes_{\red \alpha} G$; that $1\pi\rtimes_\red \lambda U$ is faithful if and only if $\pi$ is faithful and, finally, that $\alpha$ has the \wcp\ if and only if some $1\pi\rtimes \lambda U$ is faithful; in which case every such representation with $\pi$ faithful will be faithful as well.

\begin{remark}
A more common formulation of the induced regular representation on 
$L^2(G,X)=L^2(G)\otimes X$ is given pointwise by \((\widetilde\pi(a)\,\xi)(t):=\pi(\alpha_{t^{-1}}(a))\,\xi(t),\) while the group acts via \((\widetilde U_s\xi)(t):=(\lambda_s\otimes 1)\xi(t)=\xi(s^{-1}t).\) This construction is unitarily equivalent to the one described above, namely $(1\otimes\pi,\;\lambda\otimes U)\cong (\tilde\pi,\tilde U)$. Indeed, the unitary operator \((W\xi)(t):=U_t\,\xi(t)\) intertwines both representations: \(W\,\widetilde\pi(a)\,W^*=1\otimes\pi(a)\) and \(W\,(\lambda_s\otimes 1)\,W^*=\lambda_s\otimes U_s.\)
\end{remark}

The key to the main result of this section is:

\begin{lemma}\label{lemma: the key lemma}
    Let $(X,\pi,U)$ be a nondegenerate covariant representation of a \Cstar-action $\alpha$ of a group $G$ on a \Cstar-algebra $A$ and set $M:=\pi(A)'$ (the commutant). Let $\beta$ be the continuous part of $\Ad(U)|_{M}$, viewed as \Cstar-action of $G$ on $M^c$. 
    If the diagonal action $\alpha\otmax^d\beta$ of $G$ on $A\otmax M^c$ has the \wcp, then $\Ad(U)|_{M}$ and $\Ad(U)|_{\pi(A)''}$ are \Wstar-amenable.
\end{lemma}
\begin{proof}[Sketch of proof]
    The full details of this proof are deferred to Appendix~\ref{sec: deferred proofs}. The core idea is to construct an equivariant conditional expectation $P\colon (\li{G} \bo Z(M))^c\to Z(M)^c$, this suffices because $Z(M)=Z(\pi(A)'')$.
\end{proof}

In the theorem below, the implication \eqref{item: system amenable}$\Rightarrow$\eqref{item: all tp have wcp} is part of \cite[Theorem~5.16]{BssEff_amenability}.

\begin{theorem}\label{thm: equivalence of amenability for cstar actions}
    Let $\alpha$  be a \Cstar-action of a group $G$ on a \Cstar-algebra $A$.
    Using the universal unitary implementation $(A\sbe \bB(X_\alpha),U^\alpha)$, let $A'_\alpha$ be the commutant of $A$ in $\bB(X_\alpha)$, and let $\beta$ be the continuous part of $\Ad(U^\alpha)|_{A'_\alpha}$.
    In this setting, the following claims are equivalent.
    \begin{enumerate}
        \item\label{item: system amenable} $\alpha$ is \Cstar-amenable.
        \item\label{item: all tp have wcp} For every \Cstar-action $\gamma$ of $G$, the diagonal \Cstar-action $\alpha\otmax^d\gamma$ of $G$ has the \wcp.
        \item\label{item: one specific tp has the wcp} The \Cstar-action  $\alpha\otmax^d\beta$ has the $\wcp$.
    \end{enumerate}
\end{theorem}
\begin{proof}
    Since \eqref{item: all tp have wcp} trivially implies \eqref{item: one specific tp has the wcp}, it suffices to show that \eqref{item: one specific tp has the wcp} implies \eqref{item: system amenable}, but this is a straightforward consequence of Lemma~\ref{lemma: the key lemma}.
\end{proof}

\begin{remark}
    For exact groups, the equivalence \eqref{item: system amenable}$\Leftrightarrow$\eqref{item: all tp have wcp} in Theorem~\ref{thm: equivalence of amenability for cstar actions} was already known, see \cite[Theorem~5.16]{BssEff_amenability}. The argument given there relies on Haagerup's standard form of von Neumann algebras. Our proof is different and avoids the use of this machinery.
\end{remark}

\section{Amenability of Fell bundles and \Cstar-partial actions}\label{sec: amenability of Fell bundles}

When we say that $\cB=\{B_t\}_{t\in G}$ is a Fell bundle we mean that the disjoint union $\bigsqcup_{t\in G} B_t$ forms a \Cstar-algebraic bundle over $G$ in the sense of \cite[Ch.~VIII, 16.2]{FlDr88}, with fibers $B_t$.  
Following \cite{Ab03}, we denote by $C_c(\cB)$, $L^1(\cB)$, $C^*(\cB)$, $C^*_\red(\cB)$, and $\bk(\cB)$ the *-algebra of continuous cross-sections with compact support, the $L^1$ cross-sectional algebra, the full cross-sectional \Cstar-algebra, the reduced cross-sectional \Cstar-algebra, and the \Cstar-algebra of kernels of $\cB$, respectively.  

Recall that the algebra of kernels with compact support, $\bk_c(\cB)$, consists of compactly supported continuous functions $k\colon G\times G\to \cB$ with $k(s,t)\in B_{st^{-1}}$.  
It is a dense *-subalgebra of $\bk(\cB)$, fixed under the canonical action $\beta$ of $G$ on $\bk(\cB)$, which is given 
\[
   \beta_r(k)(s,t) \;=\; k(sr,tr)\,\Delta(r),
   \qquad (r,s,t\in G,\; k\in \bk_c(\cB)).
\]

As usual, the semidirect product bundle associated to a \Cstar-action $\alpha$ of $G$ on $A$ is $A\times G$, with operations
\[
   (a,r)(b,s) = (a\,\alpha_r(b),rs), 
   \qquad (a,r)^* = (\alpha_{r^{-1}}(a)^*,r^{-1}).
\]
It is customary to write $a\delta_r$ instead of $(a,r)$, so that the fiber $A\times\{r\}$ is denoted $A\delta_r$.  

In general, semidirect product bundles can be constructed from twisted partial actions on \Cstar-algebras \cite{Ex97twisted}.  
However, we shall only use \Cstar-partial actions here, as in \cite{Ab03}.

The reduced cross-sectional \Cstar-algebra of a Fell bundle was constructed by Exel and Ng in~\cite{ExNg}. 
For a Fell bundle $\cB$, we denote by 
\[
  \lambda^\cB \colon C^*(\cB) \longrightarrow C^*_\red(\cB)
\]
the canonical surjective \Star{}homomorphism given by the reduced representation~\cite[Definition~2.7]{ExNg}. 
We say that $\cB$ has the \emph{weak containment property} (\wcp) if $\lambda^\cB$ is injective, that is, if the full and reduced cross-sectional \Cstar-algebras of $\cB$ coincide. 
Exel and Ng referred to this property as \emph{amenability} of the Fell bundle, but in order to avoid confusion with other notions of amenability used in this article, we adopt the terminology \wcp\ instead (see Definition~\ref{def:FellBundleAmenable} and Remark~\ref{rem:FellBundleAmenable}).

Whenever we say that a Fell bundle has an approximation property in the sense of Exel and Ng, we mean one of those introduced in \cite[Definition~3.6]{ExNg}; by \cite[Theorem~3.9]{ExNg}, each of them implies the \wcp.  
The $L^1$-approximation property, introduced by F. Abadie in \cite[Definition~4.9(3)]{abadie1997tensorv3}, also implies the \wcp\ \cite[Theorem~4.11]{abadie1997tensorv3}, and is formally weaker than the Exel–Ng approximation properties \cite[Proposition~4.10]{abadie1997tensorv3}.  
In what follows, we shall prove that all these approximation properties are in fact equivalent.  

F. Abadie constructed a maximal tensor product of Fell bundles, denoted $\bigotimes_{\max}$, using the maximal tensor product of \Cstar-algebras \cite[Proposition~3.4 and Definition~3.8]{abadie1997tensorv3}.  
Given Fell bundles $\cA=\{A_t\}_{t\in G}$ and $\cB=\{B_t\}_{t\in H}$, their tensor product
\[
   \cA\bigotimes_{\max}\cB=\{A_s\otmax B_t\}_{(s,t)\in G\times H}
\]
is a Fell bundle over $G\times H$, where each fiber is the maximal tensor product $A_s\otmax B_t$ (of right Hilbert modules over $A_e$ and $B_e$). 
The operations are given by 
$(a\otimes b)(a'\otimes b')=aa'\otimes bb'$ and $(a\otimes b)^*=a^*\otimes b^*$, and the topology is the unique Banach bundle topology for which 
\[
   f\oslash g\colon G\times H\to \cA\bigotimes_{\max}\cB,\qquad (s,t)\mapsto f(s)\otimes g(t),
\]
is continuous for all $f\in C_c(\cA)$ and $g\in C_c(\cB)$.  
When $G=H$, we view $G$ as the diagonal of $G\times G$ and define the \emph{diagonal tensor product} as the reduction of $\cA\bigotimes_{\max}\cB$ to $G$, that is
\[
   \cA\otmax^d \cB := \{A_t\otmax B_t\}_{t\in G},
\]
so that for $f\in C_c(\cA)$ and $g\in C_c(\cB)$ one has 
\[
   (f\oslash^d g)(t) := f(t)\otimes g(t)\equiv (f\oslash g)(t,t).
\]

\medskip
Before stating our main theorem for Fell bundles, we recall the relevant approximation properties.

We use \cite[VIII~3.8]{FlDr88} to regard the multiplier algebra $\M(B_e)$ of the fiber over the unit $e\in G$ 
as the \Cstar-algebra of multipliers of $\cB$ of order $e$.  
As shown in \cite{abadie1997tensorv3}, the natural actions $\M(B_e)\times \cB \to \cB$ and 
$\cB\times \M(B_e) \to \cB$ are continuous.  
Given $\xi,\eta\in C_c(G,\M(B_e))$, we define a continuous map 
\[
   \Phi_{\xi\eta}\colon \cB \longrightarrow \cB, 
   \qquad 
   b\in B_t \;\mapsto\; 
   \Phi_{\xi\eta}(b)\equiv \xi\cdot b\cdot \eta 
   := \int_G \xi(s)^*\, b\,\eta(s^{-1}t)\, ds \;\in B_t;
\]
which is linear on each fiber.
With the norm
\[
   \|\xi\|_2 \;:=\; 
   \Big\| \int_G \xi(s)^*\xi(s)\,ds \Big\|^{1/2}
   \quad\text{on } L^2(G,\M(B_e)),
\]
one has 
\(\|\Phi_{\xi\eta}(b)\|\leq \|b\|\,\|\xi\|_2\,\|\eta\|_2.\)  
When $\xi=\eta$ we simply write $\Phi_\xi=\Phi_{\xi\xi}$.

\medskip

Let $\CLB(\cB)$ be the set of continuous maps $\Phi\colon \cB\to \cB$ 
that preserve fibers, are linear on fibers, and satisfy
\[
   \|\Phi\| := \sup\{\,\|\Phi(b)\|\colon b\in \cB,\,\|b\|\leq 1\}<\infty.
\]
A net $\{\Phi_i\}_{i\in I}\subset \CLB(\cB)$ converges to $\Phi\in \CLB(\cB)$ 
\emph{on compact slices} \cite[Definition~3.4]{ExNg} if 
\[
   \lim_i \|\Phi_i\circ f - \Phi\circ f\|_\infty = 0
   \qquad\text{for all } f\in C_c(\cB).
\]
Replacing $\|\cdot\|_\infty$ by $\|\cdot\|_1$ gives the notion of 
\emph{$L^1$-convergence on compact slices}.  
Since $\supp(\Phi_i\circ f - \Phi\circ f)\subset\supp(f)$, 
uniform convergence on compact slices implies $L^1$-convergence.

\begin{definition}[Approximation properties for Fell bundles]\label{def:approx-properties}
Let $\cB=\{B_t\}_{t\in G}$ be a Fell bundle.
\begin{enumerate}
  \item $\cB$ has the \emph{$L^1$-approximation property} (F. Abadie)
  if there exist a constant $M>0$ and nets 
  $\{\xi_i\}_{i\in I},\{\eta_i\}_{i\in I}\subset C_c(G,\M(B_e))$ 
  such that $\|\xi_i\|_2\|\eta_i\|_2\leq M$ for all $i$, 
  and $\{\Phi_{\xi_i\eta_i}\}_{i\in I}$ converges in the $L^1$-sense 
  on compact slices to the identity map $\id_\cB$.

  \item $\cB$ has the \emph{approximation property} (Exel--Ng) 
  if the nets can be chosen in $C_c(G,B_e)$ 
  and the convergence is uniform on compact slices.
\begin{enumerate}

  \item $\cB$ has the \emph{positive approximation property} 
  if, in addition, $\xi_i=\eta_i$ for all $i$.

  \item $\cB$ has the \emph{positive $1$-approximation property} 
  if, moreover, the constant $M$ can be chosen to be $1$.
  \end{enumerate}
\end{enumerate}
\end{definition}

With these definitions in place and abbreviating \textit{approximation property} to AP, we have the following chain of (trivial) implications:
\[
   \text{positive $1$-AP} \;\Rightarrow\;
   \text{positive AP} \;\Rightarrow\;
   \text{AP} \;\Rightarrow\;
   L^1\text{-AP}.
\]

We are now ready to state the main result of this section.

\begin{theorem}\label{thm: main thm about approximation properties and amenability of Fell bundles}
   Let $\cB=\{B_t\}_{t\in G}$ be a Fell bundle. The following conditions are equivalent:
   \begin{enumerate}
      \item\label{item: B has 1 pap} $\cB$ has the positive $1$-approximation property of Exel and Ng.
      \item\label{item: B has pap} $\cB$ has the positive approximation property of Exel and Ng.
      \item\label{item: B has ap} $\cB$ has the approximation property of Exel and Ng.
      \item\label{item: B has Loneap} $\cB$ has the $L^1$-approximation property of F. Abadie.
      \item\label{item: diagonal product has wcp} For every Fell bundle $\cA=\{A_t\}_{t\in G}$, the diagonal tensor product $\cA\otmax^d\cB$ has the \wcp.
      \item\label{item: beta is amenable} The canonical action $\beta$ of $G$ on $\bk(\cB)$ is \Cstar-amenable.
   \end{enumerate}
\end{theorem}

\begin{remark}
   In Definition~\ref{defi: BCAP} we introduce two more approximation properties, and in Corollary~\ref{cor:another characterization of amenability} we show that they are equivalent to the ones above.
\end{remark}

As already mentioned, we have some trivial implications, namely
\( \eqref{item: B has 1 pap}\;\Rightarrow\;\eqref{item: B has pap}
   \;\Rightarrow\;\eqref{item: B has ap}
   \;\Rightarrow\;\eqref{item: B has Loneap}\). For the remaining implications we need some lemmas.

\begin{lemma}\label{lemma: convergence on gamma suffices}
   Let $\cB=\{B_t\}_{t\in G}$ be a Fell bundle, $\{\Phi_i\}_{i\in I}$ a net in $\CLB(\cB)$, and $\Phi\in \CLB(\cB)$.  
   Assume that $M:=\sup_{i\in I}\|\Phi_i\|<\infty$ and that there exists $\Gamma\subseteq C_c(\cB)$ such that for $\Gamma':=\spn(\Gamma)$ one has $C_c(G)\Gamma'\subseteq\Gamma'$ and, for each $t\in G$, the set $\{f(t):f\in\Gamma'\}$ is dense in $B_t$.  
   If $\lim_i\|\Phi_i\circ f-\Phi\circ f\|_\infty=0$ for all $f\in C_c(\cB)$, then $\{\Phi_i\}_{i\in I}$ converges to $\Phi$ uniformly on compact slices.  
   Moreover, if in the limit we use $\|\cdot\|_1$, then $\{\Phi_i\}_{i\in I}$ $L^1$-converges to $\Phi$ on compact slices.
\end{lemma}

\begin{proof}
   By \cite[II~14.6]{FlDr88}, $\Gamma'$ is dense in $C_c(\cB)$ for the inductive limit topology, hence also dense for both $\|\cdot\|_\infty$ and $\|\cdot\|_1$.  
   Let $\|\cdot\|$ denote either of these norms. For $f\in C_c(\cB)$, $n\in\bN$, $z_1,\ldots,z_n\in\bC$, $g_1,\ldots,g_n\in\Gamma$, and $i\in I$, one has
   \[
      \|\Phi_i\circ f - \Phi\circ f\|
      \;\leq\; (M+\|\Phi\|)\Big\|f-\sum_{j=1}^n z_jg_j\Big\|
      + \sum_{j=1}^n |z_j|\;\|\Phi_i\circ g_j-\Phi\circ g_j\|.
   \]
   A standard approximation argument then yields $\lim_i\|\Phi_i\circ f - \Phi\circ f\|=0$.
\end{proof}

We continue the proof of Theorem~\ref{thm: main thm about approximation properties and amenability of Fell bundles} with the following

\begin{lemma}\label{lemma: the ap and diagonal tensor}
   Let $\cA=\{A_t\}_{t\in G}$ and $\cB=\{B_t\}_{t\in G}$ be Fell bundles.  
   If at least one of them has the $L^1$-approximation property, then so does $\cA\otmax^d\cB$.
\end{lemma}

\begin{proof}
   Assume $\cA$ has the $L^1$-approximation property.  
   For $\xi,\eta\in C_c(G,\M(A_e))$ define
   \[
      \tilde{\xi}(t) := \xi(t)\otimes 1\in \M(A_e\otmax B_e) \quad (t\in G).
   \]
   Then $\tilde{\xi}\in C_c(G,\M(A_e\otmax B_e))$ and for $c\in A_t$, $b\in B_t$ we have
   \[
      \Phi_{\tilde{\xi}\tilde{\eta}}(c\otimes b) 
      = \int_G \xi(s)^*c\eta(s^{-1}t)\otimes b \,ds 
      = \Phi_{\xi\eta}(c)\otimes b.
   \]
   Moreover, $\|\tilde{\xi}\|=\|\xi\|$ in $L^2(G,\M(A_e\otmax B_e))$.  
   For $f\in C_c(\cA)$ and $g\in C_c(\cB)$ one checks
   \[
      \Phi_{\tilde{\xi}\tilde{\eta}}(f\oslash^d g) 
      = (\Phi_{\xi\eta}\circ f)\oslash^d g,
   \]
   and hence
   \[
      \|\Phi_{\tilde{\xi}\tilde{\eta}}(f\oslash^d g) - f\oslash^d g\|_1
      \;\leq\; \|\Phi_{\xi\eta}\circ f - f\|_1 \, \|g\|_\infty.
   \]

   Let $\{\xi_i\},\{\eta_i\}\subseteq C_c(G,\M(A_e))$ witness the $L^1$-approximation property of $\cA$.  
   Then $\|\tilde{\xi}_i\|\|\tilde{\eta}_i\| = \|\xi_i\|\|\eta_i\|\leq M$, and for all $f\in C_c(\cA)$, $g\in C_c(\cB)$,
   \[
      \lim_i \|\Phi_{\tilde{\xi}_i\tilde{\eta}_i}(f\oslash^d g) - f\oslash^d g\|_1 = 0.
   \]
   Taking $\Gamma:=\{f\oslash^d g: f\in C_c(\cA),\, g\in C_c(\cB)\}$ and applying Lemma~\ref{lemma: convergence on gamma suffices}, we conclude that $\cA\otmax^d\cB$ has the $L^1$-approximation property.
\end{proof}

Since the $L^1$-approximation property implies the \wcp\ \cite[Theorem~4.11]{abadie1997tensorv3}, 
Lemma~\ref{lemma: the ap and diagonal tensor} yields the implication 
\eqref{item: B has Loneap}$\Rightarrow$\eqref{item: diagonal product has wcp} 
in Theorem~\ref{thm: main thm about approximation properties and amenability of Fell bundles}. To prove \eqref{item: diagonal product has wcp}$\Rightarrow$\eqref{item: beta is amenable}, we shall use the notion of (Morita) equivalence of Fell bundles from \cite[Definition~2.2]{AbFrrEquivalence}, together with the strong version of \cite[Definition~2.6]{AbBssFrrMorita} (for which we adopt the terminology ``strong equivalence'').

\begin{lemma}\label{lem:equivalence preserved under diagonal tensor}
Let $\cA=\{A_t\}_{t\in G}$, $\cB=\{B_t\}_{t\in G}$, and $\cC=\{C_t\}_{t\in G}$ be Fell bundles. 
If $\cA$ is the semidirect product bundle of a \Cstar-action and $\cB$ and $\cC$ are equivalent, then $\cA\otmax^d\cB$ and $\cA\otmax^d\cC$ are equivalent in the same sense.
\end{lemma}
\begin{proof}[Sketch of proof]
The complete proof is deferred to Appendix~\ref{sec: deferred proofs}. The main idea is to view $\cA\otmax^d \cB$ and $\cA\otmax^d \cC$ as Fell subbundles of $\cA\otmax^d \bL(\cX)$, where $\bL(\cX)$ is the linking bundle associated with the $\cB$--$\cC$ equivalence bundle $\cX$. The $\cA\otmax^d\cB$--$\cA\otmax^d\cC$ equivalence is implemented by a subbundle $\cA\otmax^d \cX\subset \cA\otmax^d \bL(\cX) $.
\end{proof}

\begin{corollary}
In Theorem~\ref{thm: main thm about approximation properties and amenability of Fell bundles}, the implication 
\eqref{item: diagonal product has wcp}$\;\Rightarrow\;$\eqref{item: beta is amenable} holds.
\end{corollary}

\begin{proof}
Let $\alpha$ be a \Cstar-action of $G$ on a \Cstar-algebra $A$, and let $\cA$ denote its semidirect product bundle. 
Recall that $\beta$ is the natural action of $G$ on $\bk(\cB)$, and let $\cC$ be its semidirect product bundle. 

By \cite[Theorem~3.4]{AbBssFrrMorita}, $\cB$ is equivalent to $\cC$, and hence, by Lemma~\ref{lem:equivalence preserved under diagonal tensor}, $\cA\otmax^d \cB$ is equivalent to $\cA\otmax^d \cC$. 
Since $\cA\otmax^d \cB$ has the \wcp\ (equivalently, it is $\red$-amenable in the sense of \cite[Definition~4.9]{AbFrrEquivalence}), \cite[Corollary~4.20]{AbFrrEquivalence} implies that $\cA\otmax^d \cC$ also has the \wcp. 

Now, $\cA\otmax^d \cC$ is precisely the semidirect product bundle of the diagonal action $\alpha\otmax^d \beta$. 
Thus, the reduced representation of $\cA\otmax^d \cC$ is the canonical map
\[
   C^*(\cA\otmax^d \cC) \;=\; (A\otmax^d \bk(\cB))\rtimes G
   \;\longrightarrow\;
   (A\otmax^d \bk(\cB))\rtimes_{\red} G.
\]
Therefore, $\alpha\otmax^d \beta$ has the \wcp, and the desired conclusion follows from Theorem~\ref{thm: equivalence of amenability for cstar actions}.
\end{proof}

To complete the proof of Theorem~\ref{thm: main thm about approximation properties and amenability of Fell bundles} 
it remains to show that 
\eqref{item: beta is amenable}$\;\Rightarrow\;$\eqref{item: B has 1 pap}. 
We divide the proof into two lemmas. 
The first one provides an alternative formulation of the positive approximation property of Exel and Ng. 
Note that in the statement below we do not assume $\{\|\xi_i\|_2 : i\in I\}$ to be bounded.

\begin{lemma}\label{lemma: symmetric nets}
Let $\cA=\{A_t\}_{t\in G}$ be a Fell subbundle of $\cB=\{B_t\}_{t\in G}$ such that $A_e$ is hereditary in $B_e$. 
Suppose there exists a net $\{\xi_i\}_{i\in I}\subseteq C_c(G,\M(B_e))$ such that 
$\{\Phi_{\xi_i}\}_{i\in I}$ converges to $\id_\cB$ on compact slices 
(as is the case when $\cB$ has the positive approximation property of Exel and Ng). 
Then $\cA$ has the positive $1$-approximation property.
\end{lemma}

\begin{proof}
It suffices to show that for every finite set $S\subseteq C_c(\cA)$ and $\varepsilon>0$ 
there exists $\eta\in C_c(G,A_e)$ with $\|\eta\|_2\leq 1$ such that 
\[
   \|\Phi_\eta(b)-b\| < \varepsilon
   \qquad \text{for all } b\in S(G):=\bigcup_{f\in S} f(G).
\]

Fix such $S$ and $\varepsilon$. 
Since every approximate unit of $A_e$ is a strong approximate unit for $\cA$, 
there exists a positive element $a\in A_e$ with $\|a\|<1$ such that 
\[
   \|aba-b\| < \tfrac{\varepsilon}{4} \qquad \forall\, b\in S(G).
\]
It follows that
\[
   \|a^2ba^2 - b\|
   \;\leq\; \|a(aba-b)a\| + \|aba-b\|
   \;<\; \tfrac{\varepsilon}{2}.
\]

Choose $f\in C_c(\cA)$ with $f(e)=a^2$, and for each $g\in S$ define $g'\in C_c(\cA)$ by $g'(s)=a g(s)a$. 
Set $S':=\{g' : g\in S\}\cup\{f\}$. 
By assumption, there exists $\xi\in C_c(G,\M(B_e))$ such that
\[
   \|\Phi_\xi(b)-b\| < \min\{\tfrac{\varepsilon}{2},\,1-\|a\|^2\}
   \qquad \forall\, b\in S'(G).
\]

Now define $\eta\in C_c(G,A_e)$ by $\eta(t)=a\xi(t)a$. 
Then for all $t\in G$ and $b\in B_t$, 
\[
   \Phi_\eta(b) = \int_G a\xi(s)^*\,aba\,\xi(s)\,a\,ds = a\,\Phi_\xi(aba)\,a.
\]

For $b\in S(G)$ set $b':=aba\in S'(G)$. 
Since $\|b'-b\|<\varepsilon/2$, we obtain
\[
   \|\Phi_\eta(b)-b\|
   \;\leq\; \|a\Phi_\xi(b')a - a^2ba^2\| + \|a^2ba^2-b\|
   \;<\; \|a\|^2\|\Phi_\xi(b')-b'\| + \tfrac{\varepsilon}{2}
   \;<\; \varepsilon.
\]

Finally, note that
\[
   \langle \eta,\eta\rangle = a\Phi_\xi(a^2)a,
\]
and since $a^2\in S'(G)$, we have
\[
   \|\eta\|^2
   \;\leq\; \|\Phi_\xi(a^2)\|
   \;\leq\; \|\Phi_\xi(a^2)-a^2\| + \|a^2\|
   < (1-\|a\|^2)+\|a\|^2 = 1.
\]
Thus $\eta$ has the desired properties.
\end{proof}

\begin{lemma}\label{lemma: pmap and the strong equivalente}
    The positive $1$-approximation property of Exel and Ng is preserved under strong equivalence of Fell bundles.
\end{lemma}
\begin{proof}[Sketch of proof]
The detailed computations are given in Appendix~\ref{sec: deferred proofs}. The proof relies on using a strong $\cA$--$\cB$ equivalence bundle $\cX$ to transfer the nets $\{\xi_i\}$ witnessing the positive 1-AP of $\cA$ to corresponding nets for $\cB$. Using Lemma~\ref{lemma: convergence on gamma suffices} and the properties of the inner products in $\cX$, we explicitly construct an approximating net of elements $\eta \in C_c(G, B_e)$ bounded by 1, demonstrating that the set of such approximants is nonempty.
\end{proof}

\begin{proof}[Proof of Theorem~\ref{thm: main thm about approximation properties and amenability of Fell bundles}]
As we have seen, we already have 
\eqref{item: B has 1 pap}$\Rightarrow$\eqref{item: B has pap}$\Rightarrow$\eqref{item: B has ap}$\Rightarrow$\eqref{item: B has Loneap}$\Rightarrow$\eqref{item: diagonal product has wcp}$\Rightarrow$\eqref{item: beta is amenable}.
Assume now that \eqref{item: beta is amenable} holds.
Combining Theorems~2.13 and 3.2 of \cite{ozawa2021characterizations}, we obtain that $\beta$ has Exel--Ng's approximation property for \Cstar-actions. 
By definition, this means that the semidirect product bundle $\cC$ of $\beta$ has the positive approximation property of Exel and Ng.
Recall from \cite[Theorem~3.5]{AbBssFrrMorita} that $\bk(\cB)$ contains a \Cstar-ideal $\bK$ such that the semidirect product bundle $\cA$ of the restricted (partial) action $\alpha:=\beta|_\bK$ is strongly Morita equivalent to $\cB$.
Moreover, by \cite[Proposition~4.2]{AbFrrEquivalence}, $\cA$ is a hereditary Fell subbundle of $\cC$.
Applying Lemma~\ref{lemma: symmetric nets}, we deduce that $\cA$ has the positive $1$-approximation property of Exel and Ng.
Finally, since $\cA$ is strongly equivalent to $\cB$ by \cite[Theorem~3.5]{AbBssFrrMorita}, Lemma~\ref{lemma: pmap and the strong equivalente} implies that $\cB$ itself has the positive $1$-approximation property.
\end{proof}

\begin{definition}\label{def:FellBundleAmenable}
    A Fell bundle $\cB$ is called \emph{\Cstar-amenable}  if it satisfies any (and hence all) of the equivalent conditions listed in Theorem~\ref{thm: main thm about approximation properties and amenability of Fell bundles}. 
    \vskip 0.5pc
    A \Cstar-partial action is called \emph{\Cstar-amenable} if its associated semidirect product Fell bundle is \Cstar-amenable.
\end{definition}

\begin{remark}\label{rem:FellBundleAmenable}
    In \cite{ExNg}, Exel and Ng defined a Fell bundle to be \emph{amenable} if the canonical map 
    \[
        \lambda^\cB \colon C^*(\cB) \longrightarrow C^*_\red(\cB)
    \]
    is injective. 
    Their notion of amenability therefore coincides with what we call in this paper the \emph{weak containment property} (\wcp). 
    To avoid ambiguity, we reserve the term \emph{amenability} for the stronger concept introduced in Definition~\ref{def:FellBundleAmenable}, which is expressed in terms of approximation properties and aligns more closely with the classical notion of amenability for group actions in the sense of Anantharaman--Delaroche.
\end{remark}

\begin{remark}
    Assume $\alpha$ is a \Cstar-partial action of $G$ on $A$, let $\cB$ denote its semidirect product bundle and let $\beta$ be the natural action of $G$ on $\bk(\cB)$. 
    Then $\beta$ is a Morita enveloping action of $\alpha$ \cite[Theorem~6.1]{Ab03}, and Morita enveloping actions are unique up to Morita equivalence \cite[Proposition~6.3]{Ab03}.
    In particular, if $\alpha$ is already a \Cstar-action, then $\beta$ is Morita equivalent to $\alpha$.
    Since \Cstar-amenability is invariant under Morita equivalence of \Cstar-actions \cite[Proposition~3.20]{BssEff_amenability}, Definition~\ref{def:FellBundleAmenable} is consistent with the usual notion of amenability for \Cstar-actions.
    More generally, a \Cstar-partial action $\alpha$ is \Cstar-amenable if and only if any (equivalently, all) of its Morita enveloping actions are \Cstar-amenable.

    Each of the equivalent conditions in Theorem~\ref{thm: main thm about approximation properties and amenability of Fell bundles} can be reformulated in terms of $\alpha$. 
    For instance, for \Cstar-actions the positive approximation property of Exel and Ng coincides with the AP of \cite[Definition~4.3]{BssEff_amenability}. 
    (To get the equivalence, one may apply Lemma~\ref{lemma: convergence on gamma suffices} with $\Gamma \subset C_c(\cB)$ consisting of sections of the form $t \mapsto f(t)a\delta_t$, where $f\in C_c(G)$ and $a\in A$.) 
    On the other hand, some other approximation properties -- such as the QAP of \cite[Definition~4.1]{BssEff_amenability} or the weak versions of the AP and the QAP -- do not correspond to the translated versions of the properties in Theorem~\ref{thm: main thm about approximation properties and amenability of Fell bundles}.
\end{remark}

\begin{remark}
    Our definition of amenability for Fell bundles naturally induces a notion of amenability for twisted \Cstar-actions, and even for twisted \Cstar-partial actions: given such a twisted partial action, we assign to it the corresponding (semidirect product) Fell bundle, and we then say that the twisted partial action is amenable if the associated Fell bundle is \Cstar-amenable. 

    More generally, given any construction that associates to an object $X$ a Fell bundle $F(X)$, it is natural to declare $X$ \emph{\Cstar-amenable} whenever $F(X)$ is \Cstar-amenable.
\end{remark}

There are a number of interesting consequences we think worth mentioning.

\begin{corollary}[Lemma~\ref{lemma: the ap and diagonal tensor}]
    If $\cA$ and $\cB$ are Fell bundles over the same group and one of them is \Cstar-amenable, then $\cA\otmax^d\cB$ is \Cstar-amenable. 
    In particular, amenability is preserved under taking diagonal tensor products with arbitrary Fell bundles.
\end{corollary}

The positive AP can be weakened to an equivalent (unbounded) condition.

\begin{corollary}\label{cor: unbounded pap}
     A Fell bundle $\cB=\{B_t\}_{t\in G}$ is \Cstar-amenable if and only if there exists a net $\{\xi_i\}_{i\in I}\sbe C_c(G,\M(B_e))$ such that for all $f\in C_c(\cB)$,
     \[
         \lim_i \|\Phi_{\xi_i}\circ f - f\|_\infty = 0.
     \]
\end{corollary}
\begin{proof}
    The existence of such a net is clearly weaker than the positive $1$-approximation property.  
    For the converse, apply Lemma~\ref{lemma: symmetric nets} with $\cA=\cB$.
\end{proof}

Using the corollary above one could give a (quite long but) direct proof of the following.

\begin{corollary}\label{cor: invariance under equivalence}
     The \Cstar-amenability of Fell bundles (respectively, \Cstar-partial actions) is preserved under equivalence of Fell bundles (of \Cstar-partial actions, see \cite[Definition 4.4]{Ab03}).
\end{corollary}
\begin{proof}
      Let $\cA$ and $\cB$ be weakly equivalent Fell bundles over the same group $G$, and let $\alpha$ and $\beta$ denote the respective actions on the \Cstar-algebras of kernels. 
      By \cite[Theorem 4.2]{AbBssFrrMorita}, $\alpha$ is Morita equivalent to $\beta$. 
      Combining this with \cite[Proposition 3.20]{BssEff_amenability} and Theorem~\ref{thm: main thm about approximation properties and amenability of Fell bundles}, we obtain the chain of equivalences
      \[
          \cA \text{ is \Cstar-amenable} \ \Longleftrightarrow\ \alpha \text{ is \Cstar-amenable} 
          \ \Longleftrightarrow\ \beta \text{ is \Cstar-amenable} 
          \ \Longleftrightarrow\ \cB \text{ is \Cstar-amenable}.
      \]

      The claim for \Cstar-partial actions follows since, if two \Cstar-partial actions are Morita equivalent, then their semidirect product bundles are equivalent \cite[Section 2.2.2]{AbFrrEquivalence}.
\end{proof}

\begin{corollary}\label{cor: amenability passes to certain subbundles}
    If $\cA=\{A_t\}_{t\in G}$ is a Fell subbundle of $\cB=\{B_t\}_{t\in G}$ such that $A_e$ is hereditary in $B_e$, and if $\cB$ is \Cstar-amenable, then $\cA$ is \Cstar-amenable.
\end{corollary}
\begin{proof}
    Since $\cB$ is \Cstar-amenable, it has the positive approximation property of Exel and Ng. 
    By Lemma~\ref{lemma: symmetric nets}, this implies that any Fell subbundle $\cA\subseteq \cB$ with $A_e$ hereditary in $B_e$ inherits the positive $1$-approximation property. 
    Hence $\cA$ is \Cstar-amenable by Theorem~\ref{thm: main thm about approximation properties and amenability of Fell bundles}.
\end{proof}

The next result provides a systematic way of producing amenable \Cstar-partial actions: it shows that restricting an amenable \Cstar-action to a (possibly non-invariant) \Cstar-ideal yields an amenable \Cstar-partial action.  
Moreover, up to Morita equivalence of \Cstar-partial actions, this construction accounts for all amenable \Cstar-partial actions. Indeed, by \cite{Ab03}, every \Cstar-partial action is Morita equivalent to one that is \emph{globalizable}, i.e.\ the restriction of a global \Cstar-action.  

\begin{corollary}
    Let $\alpha=(\{A_t\}_{t\in G},\{\alpha_t\}_{t\in G})$ be a \Cstar-action that is \Cstar-amenable, and let $I$ be a \Cstar-ideal of $A$ (not necessarily invariant). Then the restricted \Cstar-partial action $\alpha|_I$ (as in \cite[Proposition~4.2]{AbFrrEquivalence}) is also \Cstar-amenable.
\end{corollary}
\begin{proof}
    The semidirect product bundle $\cB_{\alpha|_I}$ of $\alpha|_I$ is a hereditary Fell subbundle of $\cB_\alpha$.  
    The claim then follows directly from Corollary~\ref{cor: amenability passes to certain subbundles}.
\end{proof}

\begin{remark}
    The corollary above reflects a more general phenomenon for Fell bundles.  
    By \cite{AbFrrEquivalence,AbBssFrrMorita}, every Fell bundle is strongly equivalent to one arising from a \Cstar-partial action, and every such partial action is Morita equivalent to the restriction of a global \Cstar-action.  
    Since \Cstar-amenability is preserved under strong equivalence of Fell bundles and under passage to enveloping (global) actions, it follows that amenability of Fell bundles and \Cstar-partial actions is entirely governed by the amenability of global \Cstar-actions.
\end{remark}

\subsection{Applications to Nuclearity of Cross-Sectional Algebras}\label{ssec: amenability and nuclearity}

As a first application of our equivalence theorem, we obtain highly non-trivial permanence results regarding the nuclearity of cross-sectional \Cstar-algebras. We start by providing a streamlined, alternative proof of a theorem of F. Abadie (see \cite[Corollary 4.13]{abadie1997tensorv3} or \cite[Corolário 3.29]{AbPhd}).
Fix a Fell bundle $\cB=\{B_t\}_{t\in G}$.

\begin{theorem}\label{thm: B amenable and Be nuclear}
    If $\cB$ is \Cstar-amenable and $B_e$ is nuclear, then $C^*(\cB)$ is nuclear.
\end{theorem}
\begin{proof}
    Let $\beta$ be the natural action of $G$ on $\bk(\cB)$ and $\cC$ its semidirect product bundle.
    Then $\cB$ is equivalent to $\cC$ and $C^*(\cC)=\bk(\cB)\rtimes_\beta G$.
    By \cite[Corollary 5.2]{Ab03}, $\bk(\cB)$ is nuclear.
    Since $\beta$ is \Cstar-amenable, \cite[Theorem 7.2]{BssEff_amenability} implies that $C^*(\cC)$ is nuclear.
    Using \cite[Theorem 4.5]{AbFrrEquivalence} we get that $C^*(\cB)$ is nuclear.
\end{proof}

\begin{theorem}\label{thm: reduced algebra nuclear implies nuclear fiber} Let $\cB$ be a Fell bundle over a group $G$.
\begin{enumerate}
    \item  If $C^*_\red(\cB)$ is nuclear, then $B_e$ is nuclear. 
    \item If $G$ inner amenable and $C^*_\red(\cB)$ is nuclear, then $\cB$ is \Cstar-amenable.
\end{enumerate}
\end{theorem}
\begin{proof}
    By \cite[Proposition 4.13]{AbFrrEquivalence}, $C^*_\red(\cB)$ is Morita equivalent to $C^*_\red(\cC)=\bk(\cB)\rtimes_{\red \beta}G$.
    Hence, $\bk(\cB)\rtimes_{\red \beta}G$ is nuclear and \cite[Theorem 3.4]{LauPat1991} implies that $\bk(\cB)$ is nuclear (unfortunately, in that article nuclear \Cstar-algebras are called \textit{amenable}).
    Corollary 5.2 of \cite{Ab03} implies $B_e$ is nuclear.
    Finally, if $G$ is inner amenable, then \cite[Corollary 6.6]{mckee2020amenable} implies $\beta$ is \Cstar-amenable.
\end{proof}

\subsection{The case of SIN groups}\label{ssec: sin groups}
For the specific case of \Cstar-actions and SIN groups, our tensorial machinery allows for a fully self-contained proof of the converse implication in Theorem~\ref{thm: reduced algebra nuclear implies nuclear fiber}, avoiding the heavy machinery of \cite{mckee2020amenable}. 

A group $G$ has small invariant neighborhoods (SIN) if there exists a basis $\mathcal{U}$ of compact neighborhoods of $e \in G$, each of which is invariant under conjugation. Every discrete group is SIN, and every SIN group is inner amenable and unimodular. We use $\lambda$ and $\rho\colon G \to \bB(\ltwo{G})$ for the left and right regular representations, respectively, and define $\omega\colon G \to \bB(\ltwo{G})$ by setting $\omega_t := \lambda_t \rho_t$. Let $\alpha$ and $\beta$ be \Cstar-actions of $G$ on $A$ and $B$, respectively.

The natural maps $\iota^{\red\alpha}\colon A \to \M(A \rtimes_{\red \alpha} G)$ and $u^{\red\alpha}\colon G \to \M(A \rtimes_{\red \alpha} G)$, alongside their full counterparts $(\iota^{\alpha},u^{\alpha})$, form canonical covariant pairs. For all $a \in A$, $s \in G$, and $f \in C_c(G,A)$, we have $\iota^{\red\alpha}(a)f = \iota^{\alpha}(a)f = af$, and $u^{\red\alpha}_s f = u^{\alpha}_s f = \tilde{\alpha}_s(f)$.

\begin{proposition}\label{prop: faithful rep of reduced crossed product}
The covariant pair $(\iota^{\red\alpha}\otimes \iota^{\red\beta}, u^{\red\alpha}\otimes^d u^{\red\beta})$ of $\alpha \otimes^d \beta$, defined by
\begin{align*}
    \iota^{\red\alpha}\otimes \iota^{\red\beta}\colon A\otimes B &\longrightarrow \M\big((A\rtimes_{\red \alpha} G)\otimes (B\rtimes_{\red \beta} G)\big), 
    & (a\otimes b)&\longmapsto \iota^{\red\alpha}(a)\otimes \iota^{\red\beta}(b),  \\
    u^{\red\alpha}\otimes^d u^{\red\beta}\colon G &\longrightarrow \M\big((A\rtimes_{\red \alpha} G)\otimes (B\rtimes_{\red \beta} G)\big), 
    & t&\longmapsto u^{\red\alpha}_t \otimes u^{\red\beta}_t,
\end{align*}
integrates to a representation that factors through a faithful representation of $(A\otimes B)\rtimes_{\red\,(\alpha\otimes^d\beta)} G$.
\end{proposition}
\begin{proof}
    Take unitary implementations $(X,\pi_A,U)$ and $(Y,\pi_B,V)$ of $\alpha$ and $\beta$, respectively. Then $\iota^{\red\alpha}$ can be viewed as 
    \[
    A \to \bB(\ltwo{G}\otimes X \otimes \ltwo{G}\otimes Y), 
    \qquad a \mapsto 1 \otimes \pi_A(a) \otimes 1 \otimes 1,
    \]
    with $u^{\red\alpha}_r = \lambda_r \otimes U_r \otimes 1 \otimes 1$.  
    Similarly, $\iota^{\red\beta}(b) = 1 \otimes 1 \otimes 1 \otimes \pi_B(b)$ and $u^{\red\beta}_r = 1 \otimes 1 \otimes \lambda_r \otimes V_r$.
    
    With this notation, we obtain
    \[
    \iota^{\red\alpha}\otimes \iota^{\red\beta}(a\otimes b) 
    = 1 \otimes \pi_A(a) \otimes 1 \otimes \pi_B(b),
    \qquad 
    (u^{\red\alpha}\otimes^d u^{\red\beta})_t 
    = \lambda_t \otimes U_t \otimes \lambda_t \otimes V_t.
    \]
    
    In the notation of Section~\ref{ssec:preliminaries}, this means
    \[
    (\iota^{\red\alpha}\otimes \iota^{\red\beta},\, u^{\red\alpha}\otimes^d u^{\red\beta}) 
    = (1\otimes (\pi_A \otimes 1 \otimes \pi_B),\, \lambda \otimes (U \otimes \lambda \otimes V)).
    \]
    The conclusion follows immediately from the fact that $\pi_A \otimes 1 \otimes \pi_B$ is a faithful representation of $A\otimes B$.
\end{proof}

\begin{proposition}\label{prop: faithful rep of full crossed product}
If $G$ is SIN, then the covariant pair $(\iota^{\red\alpha}\otmax \iota^{\red\beta}, u^{\red\alpha}\otmax^d u^{\red\beta})$ of $\alpha \otmax^d \beta$, given by
\begin{align*}
    \iota^{\red\alpha}\otmax \iota^{\red\beta}\colon A\otmax B 
        &\longrightarrow \M\big((A\rtimes_{\red \alpha} G)\otmax (B\rtimes_{\red \beta} G)\big), 
        & (a\otimes b) &\longmapsto \iota^{\red\alpha}(a)\otimes \iota^{\red\beta}(b),   \\
    u^{\red\alpha}\otmax^d u^{\red\beta}\colon G 
        &\longrightarrow \M\big((A\rtimes_{\red \alpha} G)\otmax (B\rtimes_{\red \beta} G)\big),  
        & t &\longmapsto u^{\red\alpha}_t \otimes u^{\red\beta}_t,
\end{align*}
integrates to a faithful representation of $(A\otmax B)\rtimes_{\alpha\otmax^d\beta} G$.
\end{proposition}
\begin{proof}
    Let $(X,\pi,U)$ be a covariant representation of $\alpha\otmax^d\beta$ such that the integrated form $\pi\rtimes U$ is faithful and nondegenerate.  
    Decompose $\pi$ into representations $\pi_A\colon A\to \bB(X)$ and $\pi_B\colon B\to \bB(X)$, so that $\pi(a\otimes b)=\pi_A(a)\pi_B(b)$.  
    Then $(X,\pi_A,U)$ and $(X,\pi_B,U)$ are covariant representations.  

    Using the notation of Section~\ref{ssec:preliminaries}, we construct the covariant representations 
    \[
       (\ltwo{G}\otimes X,\,1\pi_A,\,\lambda U)
       \quad\text{and}\quad
       (\ltwo{G}\otimes X,\,1\pi_B,\,\lambda U),
    \]
    which induce faithful representations
    \begin{align*}
        1\pi_A\rtimes_\red \lambda U &\colon A\rtimes_{\red\alpha}G \to \bB(\ltwo{G}\otimes X), 
        & 1\pi_B\rtimes_\red \lambda U &\colon B\rtimes_{\red\beta}G \to \bB(\ltwo{G}\otimes X).
    \end{align*}

    We conjugate the second by the unitary $W\in \bB(\ltwo{G}\otimes X)$ defined by 
    \[
       (W\xi)(r) = \Delta(r)^{-1/2} U_r\xi(r^{-1}).
    \]
    Note that $W=W^*=W^{-1}$.  
    Define $\overline{1\pi_B}(b):=W(1\pi_B)(b)W$ and $\overline{\lambda U}_r:=W(\lambda U)_r W$.  
    Then 
    \[
       \overline{1\pi_B}(b)\xi(r)=\pi_B(\beta_r(b))\xi(r),
       \qquad 
       \overline{\lambda U}_r=\rho_r\otimes 1.
    \]

    Straightforward computations show that for all $r,s\in G$, $a\in A$, and $b\in B$,
    \begin{align*}
         (\lambda U_r)\,\overline{\lambda U}_s &= \overline{\lambda U}_s\,(\lambda U_r), 
         & (\lambda U_r)\,\overline{1\pi_B}(b) &= \overline{1\pi_B}(b)\,(\lambda U_r),  \\
         (1\pi_A)(a)\,\overline{1\pi_B}(b) &= \overline{1\pi_B}(b)\,(1\pi_A)(a), 
         & (1\pi_A)(a)\,\overline{\lambda U}_s &= \overline{\lambda U}_s\,(1\pi_A)(a).
    \end{align*}

    Hence, the images of $1\pi_A\rtimes_\red \lambda U$ and 
    $\overline{1\pi_B}\rtimes_\red \overline{\lambda U}:=\Ad(W)\circ (1\pi_B\rtimes_\red \lambda U)$ commute, and we obtain a representation 
    \[
       \kappa := (1\pi_A\rtimes_\red \lambda U)\otmax (\overline{1\pi_B}\rtimes_\red \overline{\lambda U})
       \colon (A\rtimes_{\red\alpha}G)\otmax (B\rtimes_{\red \beta} G) \to \bB(\ltwo{G}\otimes X).
    \]

    Extending $\kappa$ to the multiplier algebra and composing with $\iota^{\red\alpha}\otmax \iota^{\red\beta}$ and $u^{\red\alpha}\otmax^d u^{\red\beta}$, we obtain the covariant representation $(\ltwo{G}\otimes X, \Pi, V)$ with $V_r = \omega_r\otimes U_r$.  

    Let $\mathcal{U}$ be a basis of compact invariant neighborhoods of $e$.  
    For each $u\in \mathcal{U}$ let $\indfun_u\in C_c(G)$ be the characteristic function of $u$, $m_u$ its measure, and set $\xi_u:=(m_u)^{-1/2}\indfun_u$, so that $\|\xi_u\|_2=1$.  

    For all $a\in A$, $b\in B$, $x,y\in X$, and $r\in G$ we have
    \begin{align*}
        \langle \Pi(a\otimes b)V_r (\xi_u\otimes x),\xi_u\otimes y\rangle 
        &= \int_G \langle \pi_A(a)\pi_B(\beta_s(b))U_rx,y\rangle \, \overline{\omega_r\xi_u}(s)\,\xi_u(s)\, ds\\
        &= (m_u)^{-1}\int_u \langle \pi_A(a)\pi_B(\beta_s(b))U_rx,y\rangle \, ds   \\
        &= (m_u)^{-1}\int_u \langle \pi\circ (\id_A\otimes \beta_s)(a\otimes b)\,U_rx,y\rangle \, ds .
    \end{align*}

    Now let $f\in C_c(G,A\otmax B)$.  
    Since the image of $f$ is compact,
    \[
       \lim_u \sup\{\| (\id_A\otimes \beta_s)(f(t))-f(t) \|\colon s\in u,\ t\in G \}=0.
    \]
    Writing $m_f$ for the measure of $\supp(f)$, we obtain
    \begin{multline*}
        \lim_u \Big| \langle \Pi\rtimes V (f) (\xi_u\otimes x),\xi_u\otimes y\rangle - \langle \pi\rtimes U (f)x,y\rangle \Big|\\
        \leq \lim_u \int_G (m_u)^{-1}\int_u \big| \langle \pi\circ (\id_A\otimes \beta_s)(f(r))U_rx,y\rangle - \langle \pi(f(r))U_rx,y\rangle\big|\,ds\,dr\\
        \leq \lim_u m_f \|x\|\|y\|\sup\{\| (\id_A\otimes \beta_s)(f(t))-f(t) \|\colon s\in u,\ t\in G \}  = 0.
    \end{multline*}

    Since $\|\xi_u\otimes x\|=\|x\|$, it follows immediately that 
    $\| \pi\rtimes U (f) \|\leq \| \Pi\rtimes V (f)\|$.  
    By continuity, the same inequality holds for all $f\in (A\otmax B)\rtimes_{\alpha\otmax^d\beta} G$, and hence 
    \[
       \Pi\rtimes V =\kappa\circ \big((\iota^{\red\alpha}\otmax \iota^{\red\beta})\rtimes (u^{\red\alpha}\otmax^d u^{\red\beta})\big)
    \]
    is a faithful representation.
\end{proof}

\begin{proposition}\label{prop: SIN groups and nuclearity}
    If $A$ is nuclear and $\alpha$ is \Cstar-amenable, then $A\rtimes_\red G$ is nuclear.
    The converse holds if $G$ is SIN.
\end{proposition}
\begin{proof}
   The forward implication is already known \cite[Theorem 7.2]{BssEff_amenability}, but we provide a proof here for the convenience of the reader.  
   Take any \Cstar-algebra $B$ and equip it with the trivial action $\beta$ of $G$.
   Then $\alpha\otmax^d \beta = \alpha\otimes^d \beta$ (in particular, $\alpha$) has the \wcp, and
   \begin{align*}
       (A\rtimes_{\red\alpha} G)\otmax B 
       &= (A\rtimes_\alpha G)\otmax B
        = (A\otmax B)\rtimes_{\alpha\otmax^d\beta} G
        = (A\otimes B)\rtimes_{\alpha\otimes^d\beta} G\\
       &= (A\otimes B)\rtimes_{\red \alpha\otimes^d\beta} G 
        = (A\rtimes_{\red\alpha}G)\otimes B,
   \end{align*}
   showing that $A\rtimes_{\red\alpha}G$ is nuclear.

   For the converse, let $B$ be any \Cstar-algebra equipped with the trivial action $\beta$ of $G$.
   If $\kappa\colon A\otmax B\to A\otimes B$ denotes the natural quotient map, then 
   \[
      \iota^{\red\alpha}\otmax \iota^{\red\beta} 
      = (\iota^{\red\alpha}\otimes \iota^{\red\beta})\circ \kappa
   \]
   is faithful.  
   Hence $\kappa$ is faithful, and it follows that $A$ is nuclear.

   Now let $\beta$ be an arbitrary \Cstar-action of $G$ on a \Cstar-algebra $B$. 
   By Propositions~\ref{prop: faithful rep of reduced crossed product} and~\ref{prop: faithful rep of full crossed product} we have
   \[
      (\iota^{\red\alpha}\otimes \iota^{\red\beta},u^{\red\alpha}\otimes^d u^{\red\beta})
      = (\iota^{\red\alpha}\otmax \iota^{\red\beta},u^{\red\alpha}\otmax^d u^{\red\beta}).
   \]
   Therefore $\alpha\otmax^d\beta$ has the \wcp, and Theorem~\ref{thm: equivalence of amenability for cstar actions} implies that $\alpha$ is amenable.   
\end{proof}

\section{Applications I: The B\'edos--Conti Approximation Property and Permanence}\label{sec: the BCAP}

A new approximation property for Fell bundles over discrete groups was recently introduced by B\'edos and Conti in \cite{bedos2024positiveMZ}. In this section, we apply the tensorial machinery developed in Section~\ref{sec: amenability of Fell bundles} to extend this property (BCAP) to Fell bundles over general locally compact groups. Most importantly, we prove that the BCAP is equivalent to \Cstar-amenability (and thus to Exel--Ng's AP) without requiring any nuclearity hypotheses on the unit fiber, see \cite{bedos2024positiveMZ}.

Furthermore, we leverage this equivalence to establish highly non-trivial permanence properties for \Cstar-amenability, such as the passage to closed subgroups, invariant subalgebras, and conditional expectations. These permanence results demonstrate the practical power of the BCAP as a tool to detect amenability in settings where classical spatial arguments fail.

Our approach to completely positive (cp) maps on Fell bundles follows the definitions and dilation theorems of \cite{BssFrrSeh}, where the setting is that of Fell bundles over discrete groups.  
Throughout, let $\cB=\{B_t\}_{t\in G}$ be a Fell bundle over a locally compact group $G$.  

Recall from \cite[Lemma~2.8]{AbFrrEquivalence} that for each $n\in \bN$ and $t=(t_1,\ldots,t_n)\in G^n$, the set
\[
   \bM_t(\cB)
   := \bigl\{ (b_{i,j})_{i,j=1}^n : b_{i,j}\in B_{t_i^{-1}t_j} \text{ for all } i,j \bigr\}
\]
is a \Cstar-algebra under the usual matrix operations.  
Given a *-representation $T\colon \cB\to \bB(X)$, one obtains a natural *-representation
\[
   T^t \colon \bM_t(\cB)\longrightarrow \bM(\bB(X))=\bB(X^n), 
   \qquad (b_{i,j})\longmapsto (T_{b_{i,j}})_{i,j=1}^n.
\]
Moreover, if $T|_{B_e}$ is faithful, then by the \Cstar-identity each restriction $T|_{B_t}$ is isometric, and hence $T^t$ is faithful.

The following definition extends \cite[Definition~A.2]{BssFrrSeh} to Fell bundles over locally compact groups.

\begin{definition}
    A map $\phi\colon \cB\to \bB(X)$ is said to be \emph{completely positive} (cp) if:
    \begin{enumerate}
        \item $X$ is a Hilbert space and the restriction $\phi|_{B_t}$ is linear for each fiber $B_t$.
        \item For every $n\in \bN$ and $t=(t_1,\dots,t_n)\in G^n$, the map
        \[
            \phi^t\colon \bM_t(\cB)\to \bM_n(\bB(X)), 
            \qquad (b_{i,j})\mapsto (\phi(b_{i,j})),
        \]
        is positive.
        \item For all $x,y\in X$, the function $b\mapsto \langle x,\phi(b)y\rangle$ is continuous on $\cB$.
    \end{enumerate}
    The \emph{support} of $\phi$ is
    \[
       \supp(\phi)\;:=\;\overline{\{t\in G : \phi(B_t)\neq \{0\}\}}.
    \]
    If $\|\phi|_{B_e}\|\leq 1$, then $\phi$ is called \emph{ccp} (contractive completely positive).
\end{definition}

\begin{remark}\label{rmk: basic example of cp maps}
    A basic way to construct a cp map on $\cB$ is as follows.  
    Take a *-representation $T\colon \cB\to \bB(Y)$, a Hilbert space $X$, and an operator $V\in \bB(X,Y)$.  
    Then the map $S\colon \cB\to \bB(X)$ defined by
    \[
       S_b := V^*T_bV, \qquad b\in \cB,
    \]
    is completely positive, since for every $n\in \bN$ and $t\in G^n$ we have
    \[
       S^t = (V^n)^*\, T^t \,V^n,
    \]
    where $V^n\colon X^n\to Y^n$ is the diagonal amplification of $V$.  
\end{remark}

Theorem A.4 and Remark A.5 of \cite{BssFrrSeh} give the following.

\begin{theorem}[Stinespring's Dilation Theorem]
    If $\phi\colon \cB\to \bB(X)$ is cp, then there exist a *-representation 
    $T\colon \cB\to \bB(Y)$ and an operator $V\in \bB(X,Y)$ such that 
    \[
       Y=\cspn\,T(\cB)VX
       \qquad\text{and}\qquad
       \phi(b)=V^*T_bV \quad \text{for all } b\in \cB.
    \]
    Moreover, any such $T$ is nondegenerate, and for every approximate unit 
    $\{b_i\}_{i\in I}$ of $B_e$ we have
    \[
       V^*V = \text{\sot-}\lim_i \phi(b_i), 
       \qquad 
       \|V\|^2 = \|\phi|_{B_e}\| = \lim_i \|\phi(b_i)\|.
    \]
\end{theorem}

\begin{proof}
    Let $G^d$ be the group $G$ with the discrete topology, and set 
    $\cB^d := \{B_t\}_{t\in G^d}$.  
    As sets, $\cB$ and $\cB^d$ coincide.  
    By \cite[Theorem~A.4]{BssFrrSeh}, there exist a *-representation 
    $T^d\colon \cB^d\to \bB(Y)$ and an operator $V\in \bB(X,Y)$ such that 
    $Y=\cspn\, T^d(\cB^d)VX$ and 
    $\phi(b)=V^*T^d_b V$ for all $b\in \cB^d$.

    Define $T\colon \cB\to \bB(Y)$ by $T_b:=T^d_b$.  
    To check that $T$ is \wot-continuous, it suffices (since 
    $Y=\cspn\, T^d(\cB^d)VX$) to verify that for all $c,d\in \cB$ and $x,y\in X$,  
    the function 
    \[
       b\ \mapsto\ \langle T_cVx,\,T_bT_dVy\rangle
       \ =\ \langle x,\,\phi(c^*bd)y\rangle
    \]
    is continuous in $b$.  
    This follows directly from the continuity of $\phi$.  
\end{proof}

\begin{corollary}[Integrated form of cp maps]
    If $\phi\colon \cB\to \bB(X)$ is cp, there exists a unique cp map 
    $\tilde{\phi}\colon C^*(\cB)\to \bB(X)$ such that 
    \[
       \langle x,\tilde{\phi}(f)y\rangle 
       = \int_G \langle x,\phi(f(t))y\rangle\,dt,
       \qquad f\in C_c(\cB),\; x,y\in X .
    \]
    Moreover, $\|\tilde{\phi}\|=\|\phi|_{B_e}\|$.
\end{corollary}

\begin{proof}
    By Stinespring’s theorem, $\phi(b)=V^*T_bV$ for a *-representation 
    $T\colon \cB\to \bB(Y)$ and $V\in\bB(X,Y)$. 
    Let $\tilde{T}\colon C^*(\cB)\to \bB(Y)$ be the integrated form of $T$, 
    and set $\tilde{\phi}(f)=V^*\tilde{T}(f)V$. 
    This $\tilde{\phi}$ is cp, satisfies the required identity, is unique by 
    density of $C_c(\cB)$, and  
    $\|\tilde{\phi}\|=\|V\|^2=\|\phi|_{B_e}\|$.
\end{proof}

\begin{remark}\label{rmk: inclusion of kernels suffices}
    Let $A,B,C$ be \Cstar-algebras and $\pi\colon A\to B$ a surjective *-homomorphism. 
    Suppose $\psi\colon A\to C$ is (completely) positive. Then there exists a (completely) positive map $\hat{\psi}\colon B\to C$ with $\hat{\psi}\circ\pi=\psi$ 
    iff $\ker(\pi)\subseteq \ker(\psi)$.
    
    The forward implication is clear. Conversely, if $\ker(\pi)\subseteq \ker(\psi)$, then surjectivity of $\pi$ yields the existence of a unique linear map $\hat{\psi}\colon B\to C$ with $\hat{\psi}\circ\pi=\psi$. For $b\in B^+$, take $a\in A^+$ with $\pi(a)=b$, so $\hat{\psi}(b)=\psi(a)\ge 0$, proving positivity. If $\psi$ is cp, then for all $n$ the map 
    \[
       \hat{\psi}^{(n)}\colon \bM_n(B)\to \bM_n(C),\qquad 
       (b_{i,j})\mapsto(\hat{\psi}(b_{i,j})),
    \]
    is positive since it is $\widehat{\psi^{(n)}}$. Finally, using an approximate unit of $A$, one obtains $\|\hat{\psi}\|=\|\psi\|$.
\end{remark}

\begin{corollary}\label{cor: cp map with compact support}
    If $\phi\colon \cB\to \bB(X)$ is cp with compact support, then its integrated form $\tilde{\phi}\colon C^*(\cB)\to \bB(X)$ factors through a cp map $\tilde{\phi}_\red\colon C^*_\red(\cB)\to \bB(X)$ via the reduced representation $\lambda^\cB\colon C^*(\cB)\to C^*_\red(\cB)$. Moreover, $\|\tilde{\phi}_\red\|=\|\phi|_{B_e}\|$.
\end{corollary}

\begin{proof}
    Let $T\colon \cB\to \bB(Y)$ be a dilation of $\phi$ with $V\in \bB(X,Y)$. 
    Choose a compact neighborhood $N$ of $e$, $f\in C_c(G)$ with $f(t^{-1}s)=1$ for $(t,s)\in \supp(\phi)\times N$, and $g\in C_c(G)^+$ supported in $N$ with $\int_G g=1$. 
    Define $P,Q\colon X\to L^2(G)\otimes Y$ by $Px(t)=f(t)Vx$, $Qx(t)=g(t)Vx$. 

    Recall from \cite{ExNg} that the integrated form of $S\colon \cB\to \bB(L^2(G)\otimes Y)$, $b\mapsto \lambda_t\otimes T_b$, factors through a *-representation $\hat{S}\colon C^*_\red(\cB)\to \bB(L^2(G)\otimes Y)$ with $\hat{S}\circ \lambda^\cB=\tilde{S}$. 
    For $h\in C_c(\cB)$, $x,y\in X$, and $t\in G$, 
    \[
       \int_G g(s)f(t^{-1}s)\,ds\,\langle x,\phi(h(t))y\rangle=\langle x,\phi(h(t))y\rangle,
    \]
    because  $\langle x,\phi(h(t))y\rangle\neq 0$ implies $t\in\supp(\phi)$. 
    Hence
    \[
       \langle x,P^*\hat{S}_{\lambda^\cB(h)}Qy\rangle
       =\int_G\!\int_G g(s)f(t^{-1}s)\langle x,\phi(h(t))y\rangle\,dsdt
       =\langle x,\tilde{\phi}_h y\rangle,
    \]
    and it follows that $P^*\hat{S}_{\lambda^\cB(h)}Q=\tilde{\phi}_h$ for all $h\in C^*(\cB)$; which yields $\ker(\lambda^\cB)\subseteq \ker(\tilde{\phi})$. 
    By Remark~\ref{rmk: inclusion of kernels suffices}, $\tilde{\phi}$ factors through a cp map $\tilde{\phi}_\red$ with $\|\tilde{\phi}_\red\|=\|\phi|_{B_e}\|$.
\end{proof}

\subsection{Completely positive maps between Fell bundles}

Let $\cA=\{A_t\}_{t\in G}$ and $\cB=\{B_t\}_{t\in G}$ be Fell bundles over the same group.

\begin{definition}
    We say that $\phi \colon \cA\to \cB$ is completely positive (cp) if the following conditions hold.
    \begin{enumerate}
        \item $\phi$ is continuous.
        \item\label{item: preserves fibers} For all $t\in G,$ $\phi(A_t)\sbe B_t$ and $\phi|_{A_t}$ is linear.
        \item For all $n\in \bN$ and $t\in G^n$, $\phi^t\colon \bM_t(\cA)\to \bM_t(\cB)$ given by $\phi^t((a_{i,j})_{i,j=1}^n)= (\phi(a_{i,j}))_{i,j=1}^n$,  is positive.
    \end{enumerate}
    The support of $\phi$, $\supp(\phi),$ is the closure of $\{t\in G\colon \phi(A_t)\neq \{0\}\}$.
     If $\|\phi|_{B_e}\|\leq 1$, we say that $\phi$ is ccp (contractive and completely positive).
\end{definition}

\begin{remark}\label{rmk: positivity of maps between bundles}
    For a continuous function $\phi\colon \cA\to \cB$ such that for all $t\in G$, $\phi(A_t)\sbe B_t$ and $\phi|_{A_t}$ is linear, the following are equivalent:
    \begin{enumerate}
        \item $\phi$ is cp.
        \item For every *-representation $T\colon \cB\to \bB(X),$ $T\circ \phi$ is cp.
        \item There exists a *-representation $T\colon \cB\to \bB(X)$ such that $T|_{B_e}$ is faithful and $T\circ\phi$ is cp.
    \end{enumerate}
\end{remark}

\begin{remark}
    If $\phi\colon \cA\to \cB$ is cp, then $\phi|_{A_e}$ is cp and, consequently, it is bounded.
    Combining the previous Remark with the dilation Theorem one gets that for all $a\in \cA,$ $\|\phi(a)\|\leq \|\phi|_{A_e}\|\|a\|$.
\end{remark}

\begin{theorem}\label{thm: integrated form of cp maps between bundles}
    If $\phi\colon \cA\to \cB$ is cp, then there exists a cp map $\tilde{\phi}\colon C^*(\cA)\to C^*(\cB)$ such that $\tilde{\phi}_f= \phi\circ f$, for all $f\in C_c(\cA)$.
    Besides, $\|\tilde{\phi}\|\leq \|\phi|_{A_e}\|$.
    In case $\phi$ has compact support, there exists a (``wrong way'') cp map $\hat{\phi}\colon C^*_\red(\cA)\to C^*(\cB)$ such that $\hat{\phi}\circ \lambda^\cA = \tilde{\phi}$.
    Moreover, $\|\hat{\phi}\|\leq \|\phi|_{A_e}\|$.
\end{theorem}
\begin{proof}
    Let $T\colon \cB\to \bB(X)$ be a *-representation with faithful integrated form $\tilde{T}\colon C^*(\cB)\to \bB(X).$
    Then $T\circ\phi$ is cp and it has an integrated form $\widetilde{T\circ \phi}$.
    For all $f\in C_c(\cA)$, $\|\phi\circ f\|_{C^*(\cB)}=\|\tilde{T}_{\phi\circ f}\| = \| \widetilde{T\circ \phi}_f \|\leq \|T\circ \phi|_{B_e}\|\|f\|_{C^*(\cA)}$.
    Hence, $C_c(\cA)\to C^*(\cB),$ $f\mapsto \phi\circ f,$ is continuous with respect to the \Cstar-norms.
    The unique continuous extension to $\tilde{\phi}\colon C^*(\cA)\to C^*(\cB)$ is cp because $\tilde{T}\circ \tilde{\phi}=\widetilde{T\circ \phi}$ is cp.

    In case $\phi$ has compact support, by Corollary~\ref{cor: cp map with compact support},  for all $f\in C^*(\cA)$ we have $\| \tilde{\phi}_f \|= \|\tilde{T}\circ \tilde{\phi}_f\|=\|\widetilde{T\circ \phi}_f\|\leq \|T\circ \phi|_{A_e}\|\|\lambda^\cA_f\|\leq \|\phi|_{A_e}\|\|\lambda^\cA_f\|.$
    The rest of the proof follows from Remark~\ref{rmk: inclusion of kernels suffices}.
\end{proof}

\subsubsection{Characterization of completely positive maps between semidirect product bundles}
Assume that $\alpha$ and $\beta$ are \Cstar-actions of $G$ on $A$ and $B$, respectively.
To specify a function $P\colon \cA\to \cB$ between the respective semidirect product bundles such that $P(A\delta_t)\sbe B\delta_t$, for all $t\in G$, is to specify a set of functions $\{P_t\colon A\to B\}_{t\in G}$ ($P(a\delta_t)=P_t(a)\delta_t$).
In Corollary~\ref{cor: equivariant cp map gives cp map} we give a general example where $P_t=P_e$ for all $t\in G,$ but in that case $P$ has compact support if and only if $P=0.$

\begin{lemma}\label{lemma: characterization of cp map between semidirect product bundles}
    With the notation above, $P$ is cp if and only if each one of the $P_t'$s is linear and for all $n\in \bN,$ $t_1,\ldots,t_n\in G$ and $a_1,\ldots,a_n\in A,$ $(\beta_{t_i}(P_{t_i^{-1}t_j}(\alpha_{t_i}^{-1}(a_i^*a_j))))_{i,j=1}^n$ is a positive matrix of $\bM_n(B).$
\end{lemma}
\begin{proof}
    Fix $n\in \bN$ and $t=(t_1,\ldots,t_n)\in G^n$. It is easy to check that
    \begin{align*}
       \pi_\alpha^t\colon & \bM_n(A)\to \bM_t(\cA) & \pi_\alpha^t((a_{i,j})_{i,j=1}^n) & = (\alpha_{t_i}^{-1}(a_{i,j})\delta_{t_i^{-1}t_j})_{i,j=1}^n
    \end{align*}
    is a *-isomorphism.

    Each positive $a\in \bM_n(A)$ is a sum of $n$ elements of the form $(a_i^*a_j)_{i,j=1}^n$ and
    \[ (\pi_\beta^t)^{-1}\circ P^t\circ\pi_\alpha^t( (a_i^*a_j)_{i,j=1}^n ) = (\pi_\beta^t)^{-1}(P_{t_i^{-1}t_j}(\alpha_{t_i}^{-1}(a_i^*a_j))\delta_{t_i^{-1}t_j})_{i,j=1}^n  =(\beta_{t_i}(P_{t_i^{-1}t_j}(\alpha_{t_i}^{-1}(a_i^*a_j))))_{i,j=1}^n . \]
    Hence, $P^t$ is positive if and only if for all $n\in \bN, $ $t_1,\ldots t_n\in G$ and $a_1,\ldots,a_n\in A$, the last matrix on the right above is positive.
\end{proof}

\begin{corollary}\label{cor: equivariant cp map gives cp map}
    If $P\colon A\to B$ is an equivariant cp map, then $\phi\colon \cA\to \cB$ given by $\phi(a\delta_s)=P(a)\delta_s$, is cp. 
    In addition, $\|\phi|_{A\delta_e}\|=\|P\|$.
\end{corollary}
\begin{proof}
    With the notation of Lemma~\ref{lemma: characterization of cp map between semidirect product bundles}, $P_t=P$ for all $t\in G$ and 
    \[ (\beta_{t_i}(P_{t_i^{-1}t_j}(\alpha_{t_i}^{-1}(a_i^*a_j))))_{i,j=1}^n =(P(a_i^*a_j))_{i,j=1}^n\]
    is a positive matrix.
    For all $a\in A$, $\|\phi(a\delta_e)\|=\|P(a)\|\leq \|a\|=\|a\delta_e\|$.
    Hence, $\|\phi|_{A\delta_e}\|=\|P\|$.
\end{proof}

\subsection{Equivalence of the BCAP and \Cstar-amenability}

\begin{definition}\label{defi: BCAP}
A Fell bundle $\cB=\{B_t\}_{t\in G}$ has the approximation property of Bédos and Conti (BCAP) if there exists a net of ccp maps with compact support $\{\Phi_i\colon \cB\to \cB\}_{i\in I}$ that converges to $\id_\cB$ uniformly on compact slices.
The BCAP$^1$ differs only in the convergence condition: it requires $L^1$-convergence on compact slices.
\end{definition}

In the definition above the nets are bounded, so BCAP implies BCAP$^1$.
To prove that both are equivalent to the amenability of the bundle we show two results that should be true if the equivalence holds (Propositions~\ref{prop: some easy properties} and~\ref{prop: BCAP and diagonal tensor product}).
For Fell bundles over discrete groups, the Proposition below follows from the results of \cite{bedos2024positiveMZ}.

\begin{proposition}\label{prop: some easy properties}
    The following claims hold for every Fell bundle $\cB=\{B_t\}_{t\in G}$.
    \begin{enumerate}
        \item If $\cB$ \Cstar-amenable, then $\cB$ has the BCAP.
        \item If $\cB$ has the BCAP$^1$, then $\cB$ has the \wcp.
    \end{enumerate}
\end{proposition}
\begin{proof}
    Let $\{\xi_i\}_{i\in I}\sbe C_c(G,B_e)$ be a net witnessing the positive $1$-approximation property of $\cB$.
    By definition, $\{\Phi_{\xi_i}\colon \cB\to \cB\}_{i\in I}$ is a net of continuous maps that converges to $\id_\cB$ on compact slices.
    To prove that each $\phi_{\xi_i}$ is cp we take a *-representation $T\colon \cB\to \bB(X)$ with $T|_{B_e}$ faithful.
    Define $V_i\colon X\to \ltwo{G}\times X$ by $V_ix(t)=T_{\xi_i(t)}a$ and $S\colon \cB\to \bB(\ltwo{G}\otimes Y)$  by $S(b\in B_t)=\lambda_t\otimes T_b$.
    For all $x,y\in X,$ $t\in G$ and $b\in B_t$ we have
    \begin{align*}
        \langle x,V_i^* S_b V_i y\rangle 
        & = \langle V_ix,S_b V_i y\rangle
         = \int_G\langle T_{\xi_i(s)}x,T_b T_{\xi_i(\tmu s)}y\rangle ds
         = \int_G \langle x,T_{\xi_i(s)^* b \xi_i(\tmu s)}y\rangle ds\\
        & = \langle x,T\circ \Phi_{\xi_i}(b)y\rangle.
    \end{align*}
    Then $V_i^* S_b V_i = T\circ \Phi_{\xi_i}(b)$ and Remark~\ref{rmk: positivity of maps between bundles} implies that $\Phi_{\xi_i}$ is cp.
    Notice that $\|V_i\|\leq 1$, so $\|\Phi_{\xi_i}|_{B_e}\|=\|T\circ \Phi_{\xi_i}|_{B_e}\| = \|V_i ^*S|_{B_e}V_i\|\leq 1$.

    Say $\Phi_{\xi_i}(b)\neq 0$ with $b\in B_t$.
    By the definition of $\Phi_{\xi_i}(b)$, there exists $s\in \supp(\xi_i)$ such that $\tmu s\in \supp(\xi_i)$.
    This yields $\supp(\Phi_{\xi_i})\sbe \supp(\xi_i)\supp(\xi_i)^{-1}$; which completes the proof of the first claim.

    Assume that $\cB$ has the BCAP$^1$ and take a net $\{\Phi_i\}_{i\in I}$ witnessing such property.
    By Theorem~\ref{thm: integrated form of cp maps between bundles}, for all $f\in C_c(\cB)$ we have $\|\Phi_i\circ f\|_{C^*(\cB)}\leq \|f\|_{C^*_\red(\cB)}$.
    In addition, $\|\Phi_i\circ f - f\|_{C^*(\cB)}\leq \|\Phi_i\circ f-f\|_1$.
    Hence, $\|f\|_{C^*(\cB)}=\lim_i \|\Phi_i\circ f\|_{C^*(\cB)} \leq \|f\|_{C^*_\red(\cB)}\leq \|f\|_{C^*(\cB)}$ and this implies that $\lambda^\cB\colon C^*(\cB)\to C^*_\red(\cB)$ is an isometry.
\end{proof}

\begin{lemma}
    If $\cA=\{A_t\}_{t\in G}$, $\cB=\{B_t\}_{t\in G}$ and $\cC=\{C_t\}_{t\in H}$ are Fell bundles and $\Phi\colon \cA\to \cB$ is a cp map, then there exists a cp map $\Phi\bigotimes_{\max}\id_\cC\colon \cA\bigotimes_{\max}\cC\to \cB\bigotimes_{\max}\cC$ such that $\Phi\bigotimes_{\max}\id_\cC (a\otimes c)=\Phi(a)\otimes c$.
    In addition, $\|\Phi\bigotimes_{\max}\id_\cC|_{A_e\otmax B_e}\|=\|\Phi|_{A_e}\|$.
\end{lemma}
\begin{proof}
    The integrated form $\tilde{\Phi}\colon C^*(\cA)\to C^*(\cB)$ induces a cp map 
    \[\tilde{\Phi}\otmax \id \colon C^*(\cA)\otmax C^*(\cC)\to C^*(\cB)\otmax C^*(\cC)\]
    such that for all $f\in C_c(\cA)$ and $g\in C_c(\cC)$, $(\tilde{\Phi}\otmax \id)(f\otimes g) = (\Phi\circ f)\otimes g$.
    By \cite[Proposition 4.2]{abadie1997tensorv3}, there exist \Cstar-isomorphisms 
    \begin{align*}
        \mu & \colon C^*(\cA)\otmax C^*(\cC)\to C^*(\cA\bigotimes_{\max} \cC) & \nu & \colon C^*(\cB)\otmax C^*(\cC)\to C^*(\cB\bigotimes_{\max} \cC)
    \end{align*}
    such that for all $f\in C_c(\cA)$, $g\in C_c(\cB)$ and $h\in C_c(\cC)$, $\mu(f\otimes h)=f\oslash h$ and $\nu (g\otimes h)=g\oslash h$.
    The cp map 
    \begin{align*}
       \Psi:= \nu\circ (\tilde{\Phi}\otmax \id )\circ \mu^{-1} \colon C^*(\cA\bigotimes_{\max} \cC)\to C^*(\cB\bigotimes_{\max} \cC)
    \end{align*}
    is the unique continuous linear function such that for all $f\in C_c(\cA)$ and $g\in C_c(\cC)$, $\Psi(f\oslash g) = (\Phi\circ f)\oslash g$.

    Take a nondegenerate *-representation $T\colon \cB\bigotimes_{\max} \cC\to \bB(X)$ whose restriction to any fiber is isometric.
    If $\tilde{T}\colon C^*(\cB\bigotimes_{\max} \cC)\to \bB(X)$ is the integrated form of $T$, then $\tilde{T}\circ \Psi$ is cp and it has a Stinespring dilation $\tilde{S}\colon C^*(\cA\bigotimes_{\max} \cC)\to \bB(Y)$ with associated map $V\in \bB(X,Y)$.
    Naturally, we name $S\colon \cA\bigotimes_{\max} \cC\to \bB(Y)$ the disintegrated form of $\tilde{S}$.
    Using \cite[VIII 11.8]{FlDr88} and the fact that $\Psi(f\oslash g) = (\Phi\circ f)\oslash g$ get get $V^*S_{f(s)\otimes g(t)}V = T_{\Phi(f(s))\otimes g(t)}$, for all $f\in C_c(\cA)$, $g\in C_c(\cC)$ and $(s,t)\in G\times H$.
    This implies $V^*S(A_s\otmax C_t)V\sbe T(B_s\otmax C_t)$, then there exists a function $\Phi\bigotimes_{\max}\id_\cC\colon \cA\bigotimes_{\max}\cC\to \cB\bigotimes_{\max}\cC$  such that $(\Phi\bigotimes_{\max}\id_\cC)(A_s\otmax C_t)\sbe B_s\otmax C_t$ and $T\circ \Phi (z) =V^* S_z V,$ for all $(s,t)\in G\times H$ and $z\in \cA\bigotimes_{\max}\cC$.
    This implies $(\Phi\bigotimes_{\max}\id_\cC)(a\otimes c) = \Phi(a)\otimes c$ and $\| \Phi (z) \|\leq \|V\|^2\|z\|=\|\tilde{\Phi}\|\|z\|=\|\Phi|_{A_e}\|\|z\|$.
    We have $(\Phi\bigotimes_{\max}\id_\cC)(f\oslash g) = (\Phi\circ f)\oslash g$, so \cite[II 13.16]{FlDr88} implies $\Phi\bigotimes_{\max}\id_\cC$ is continuous and it is cp because $T\circ \Phi (z) =V^* S_z V$ (see Remark~\ref{rmk: positivity of maps between bundles}).
    Finally, $\|(\Phi\bigotimes_{\max}\id_\cC)|_{A_e\otmax C_e}\| = \|\Phi|_{A_e}\otmax \id_{C_e}\|=\|\Phi|_{A_e}\|$.
\end{proof}

\begin{proposition}\label{prop: BCAP and diagonal tensor product}
    If $\cA=\{A_t\}_{t\in G}$ and $\cB=\{B_t\}_{t\in G}$ are Fell bundles and at least one of them has the BCAP$^1$, then  $\cA\otmax^d\cB$ has the BCAP$^1$. 
\end{proposition}
\begin{proof}
    Let $\{\Phi_i\}_{i\in I}$ be a net of maps witnessing the BCAP$^1$ of $\cA$ and define 
    \[\Phi_i\otmax^d \id_\cB\colon \cA\otmax^d\cB\to \cA\otmax^d\cB\]
    as the restriction of $\Phi_i\bigotimes_{\max}\id_\cB$ to the reduction $\cA\otmax^d\cB = (\cA\bigotimes_{\max}\cB)_G$.
    Then $\{\Phi_i\otmax^d \id_\cB\}_{i\in I}$ is a net of cp maps and $\forall\, i\in I$, $\|(\Phi_i\otmax^d \id_\cB)|_{A_e\otmax B_e}\|=\|(\Phi_i\bigotimes_{\max}\id_\cB)|_{A_e\otmax B_e}\|=\|\Phi_i\|\leq 1$.

    For all $f\in C_c(\cA)$ and $g\in C_c(\cB)$, $(\Phi_i\otmax^d \id_\cB)\circ (f\oslash^d g) = (\Phi_i\circ f)\otmax^d g$.
    Hence, 
    \[\| (\Phi_i\otmax^d \id_\cB)\circ (f\oslash^d g) - f\oslash^d g\|_1 \leq  \|\Phi_i\circ f - f\|_1 \|g\|_\infty \]
    and taking limit in $i$ it follows that $\lim_i\| (\Phi_i\otmax^d \id_\cB)\circ (f\oslash^d g) - f\oslash^d g\|_1=0$.
    Taking $\Gamma:=\spn \{f\oslash^d g\colon f\in C_c(\cA),\ g\in C_c(\cB)\}$ and using Lemma~\ref{lemma: convergence on gamma suffices} we get that  $\cA\otmax^d\cB$ has the BCAP$^1$.
\end{proof}

\begin{corollary}\label{cor:another characterization of amenability}
    For Fell bundles, \Cstar-amenability is equivalent to the BCAP and to the BCAP$^1$.
\end{corollary}
\begin{proof}
    By Proposition~\ref{prop: some easy properties}, \Cstar-amenability implies the BCAP and right after Definition~\ref{defi: BCAP}  we noticed that BCAP implies BCAP$^1$.
    Assume that the Fell bundle $\cB$ over $G$ has the BCAP$^1$.
    Combining Propositions~\ref{prop: some easy properties} and~\ref{prop: BCAP and diagonal tensor product} we get that for any other Fell bundle $\cA$ over $G$, $\cA\otmax^d\cB$ has the \wcp,
    Theorem~\ref{thm: main thm about approximation properties and amenability of Fell bundles} implies $\cB$ is \Cstar-amenable.
\end{proof}

\begin{theorem}[Permanence under restriction to closed subgroups]\label{thm: amenability and reductions}
    If the Fell bundle $\cB=\{B_t\}_{t\in G}$ is \Cstar-amenable and $H\sbe G$ is a closed subgroup, then the restricted Fell bundle $\cB_H=\{B_t\}_{t\in H}$ is \Cstar-amenable.
\end{theorem}
\begin{proof}
    Let $\{\Phi_i\colon \cB\to \cB\}_{i\in I}$ be a net of ccp maps witnessing the BCAP of $\cB$.
    Notice that $\Phi_i(\cB_H)\sbe \cB_H$ and that for all $n\in \bN$ and $t\in H^n,$ $\bM_t(\cB)=\bM_t(\cB_H)$.
    By  \cite[II 14.8]{FlDr88}, $\{f|_H\colon f\in C_c(\cB)\}=C_c(\cB_H)$.
    Then $\{\cB_H\to \cB_H, b\mapsto \Phi_i(b)\}_{i\in I}$ is a net witnessing the BCAP of $\cB_H$.
\end{proof}

\begin{corollary}
    If $\alpha =(\{A_t\}_{t\in G},\{\alpha_t\}_{t\in G})$ is a \Cstar-amenable \Cstar-partial action and $H\sbe G$ is a closed subgroup, then $\alpha|_H:=(\{A_t\}_{t\in H},\{\alpha_t\}_{t\in H})$ is a \Cstar-amenable \Cstar-partial action.
\end{corollary}
\begin{proof}
    If $\cB_\alpha$ denotes the semidirect product bundle of $\alpha$, then $(\cB_\alpha)_H = \cB_{\alpha|_H}$ and the rest follows from the previous Theorem.
\end{proof} 

Conditional expectations for Fell bundles over discrete groups were defined by Exel \cite{ExlibroAMS}.
Recently, in \cite{abadie2024conditional}, F. Abaide extended this notion to to Fell bundles over locally compact groups and proved that all conditional expectations are completely positive and contractive.
The following result implies that if $E\colon \cB\to \cA$ is a conditional expectation and $\cB$ has the approximation property, then $\cA$ has the approximation property.

\begin{corollary}\label{cor: quasi cond expectations}
    If $\cA=\{A_t\}_{t\in G}$ is a Fell subbundle of $\cB=\{B_t\}_{t\in G}$, $\cB$ is \Cstar-amenable and there exists net of ccp maps $\{\Phi_i\colon \cB\to \cA\}_{i\in I}$ such that $\{\Phi_i|_\cA\}_{i\in I}$ converges to $\id_\cA$ on compact slices, then $\cA$ is \Cstar-amenable.
\end{corollary}
\begin{proof}
    It suffices to show that given $F\sbe C_c(\cA)$ finite and $\varepsilon>0$, there exists a ccp map $\Psi\colon \cA\to \cA$ with compact support such that for all $f\in F$, $\|\Phi\circ f-f\|_\infty <\varepsilon$.
    Fix $F$ and $\varepsilon$.
    There exists a ccp map with compact support $\Psi\colon \cB\to \cB$ such that $\forall\ f\in F$, $\|\Psi\circ f-f\|_\infty <\varepsilon/2.$
    There is also an $i\in I$ such that for all $f\in F$, $\|\Phi_i\circ f - f\|_\infty<\varepsilon/2$.
    Notice that $\Psi_i:=\Phi_i\circ \Psi|_{\cA}\colon \cA\to \cA$ is ccp and for all $f\in F$,
    \[ \| \Psi_i \circ f - f \|_\infty \leq \| \Phi_i\circ (\Psi\circ f - f) \| + \|\Phi_i\circ f - f\|_\infty \leq \| \Psi\circ f-f \|_\infty + \|\Phi_i\circ f - f\|_\infty<\varepsilon.\]
    Moreover, $\supp(\Psi_i)\sbe \supp(\Psi)$ is compact.
\end{proof}

\begin{theorem}[Permanence for invariant \Cstar-subalgebras]\label{thm: permanence invariant subalgebras}
    If $\alpha$ is a \Cstar-amenable action of $G$ on $A$, $B\sbe A$ is an invariant \Cstar-subalgebra and any of the conditions below holds, then $\alpha|_B$ is \Cstar-amenable.
    \begin{enumerate}
        \item There exists an equivariant conditional expectation $E\colon A\to B$ or, more generally, there exists a net of ccp equivariant maps $\{E_i\colon A\to B\}_{i\in I}$ such that $\{E_i|_B\}_{i\in I}$ converges pointwise in norm to $\id_B$.
        \item $B$ is hereditary in $A$.
    \end{enumerate}
\end{theorem}
\begin{proof}
    Let $\cA$ and $\cB$ be the semidirect product bundles of $\alpha$ and $\beta$, respectively.
    Define for each $i\in I$, $\Phi_i\colon \cA\to \cB$  by $\Phi_i(a\delta_t)=E_i(a)\delta_t$.
    Corollary~\ref{cor: equivariant cp map gives cp map} implies that $\Phi_i$ is ccp and because $\|E_i\|\leq 1$ for all $i\in I$, $\{E_i\}_{i\in I}$ converges to $\id_B$ uniformly on compact subsets of $B$.
    Let $\Gamma\sbe C_c(\cB)$ be the set of sections of the form $t\mapsto g(t)b\delta_t$ with $g\in C_c(G)$ and $b\in B$.
    For all $f\in \Gamma$, $\lim_i \|\Phi_i\circ f - f\|_\infty = 0$.
    Using Lemma~\ref{lemma: convergence on gamma suffices} we get that $\{\Phi_i|_\cB\}_{i\in I}$ converges uniformly to $\id_\cB$ on compact slices.
    Corollary~\ref{cor: quasi cond expectations} implies $\cB$ is amenable.
    The claim for the case where $B$ is hereditary in $A$ follows from Lemma~\ref{lemma: symmetric nets}.
\end{proof}

\section{Applications II: Permanence under Reductions and Quotient Bundles}\label{sec: reductions and quotients}

In the classical setting of group actions, permanence properties of amenability (such as restriction to subgroups or passage to quotients by normal subgroups) often rely heavily on spatial arguments, existence of conditional expectations, or exactness assumptions. For general Fell bundles, these traditional techniques break down entirely, particularly because the unit fiber $B_e$ need not be nuclear and exactness is not guaranteed.

In this section, we overcome these obstructions by deploying the BCAP developed in Section~\ref{sec: the BCAP} alongside integration of *-representations and regional topology techniques \cite{FlDr88}. We fix a Fell bundle $\cB=\{B_t\}_{t\in G}$ over a group $G$ and a closed subgroup $H\sbe G$. Our goal is to systematically study when the \wcp, the \Cstar-amenability, and the nuclearity of cross-sectional \Cstar-algebras pass from $\cB$ to its reduction $\cB_H:=\{B_t\}_{t\in H}$, and, when $H$ is normal, to the quotient bundles over $G/H$.

\subsection{Reductions to open subgroups and full cross-sectional \Cstar-algebras}
Assume $H$ is open in $G$ and think $C_c(\cB_H)\sbe C_c(\cB)$ by extending each $f\in C_c(\cB_H)$ as $0$ outside $H.$
This gives $L^1(\cB_H)\sbe L^1(\cB)$ and, passing to enveloping \Cstar-algebras, we get a *-homomorphism $\iota\colon C^*(\cB_H)\to C^*(\cB)$.
To prove $\iota$ is injective we take a nondegenerate *-representation $T\colon C^*(\cB_H)\to \bB(X)$ with faithful integrated form $\tilde{T}\colon C^*(\cB_H)\to \bB(X).$
By \cite[Theorem 2.1]{Ferraro_2024}, $T$ is $\cB-$positive and it can be induced to a *-representation $S\colon \cB\to \bB(Y)$ by the process of \cite[Ch. XI]{FlDr88}.
Since $G/H$ is discrete, we can use \cite[XI 14.21]{FlDr88} to deduce that $T$ is a subrepresentation of $S|_{\cB_H}$.
Consequently, $\tilde{T}|_{C_c(\cB_H)}$ is a subrepresentation of $\tilde{S}\circ \iota|_{C_c(\cB_H)}$ and this implies that $\iota$ is an isometry because, for all $f\in C_c(\cB_H),$ $\|\iota(f)\|\leq \|f\| = \|\tilde{T}_f\|\leq \|\tilde{S}_{\iota(f)}\|\leq \|\iota(f)\|.$
From now on we think $C^*(\cB_H)\sbe C^*(\cB)$ and omit any mention of $\iota$.

Notice that for all $f,g\in C_c(\cB_H)$ and $h\in C_c(\cB)$  we have $f*h*g = f* h|_H * g.$
Since $C_c(\cB_H)$ contains an approximate unit of $C^*(\cB_H),$ the restriction map $E_c\colon C_c(\cB)\to C_c(\cB_H),$ $E_c(h):= h|_H,$ is a contraction with respect to the \Cstar-norm.
By Tomiyama's Theorem \cite[Theorem 1.5.10]{BrownOzawa}, the continuous extension $E\colon C^*(\cB)\to C^*(\cB_H)$ of $E_c$ is a conditional expectation.

\begin{remark}
    If $C^*(\cB)$ is nuclear, then so is $C^*(\cB_H)$.
\end{remark}

\subsection{Reductions to open subgroups and reduced cross-sectional \Cstar-algebras}
We continue assuming that $H$ is open in $G$ and denote $\lambda^H$ and $\lambda^G$ the respective left regular representations.
To prove that $C_c(\cB_H)\sbe C_c(\cB)$ induces and inclusion of \Cstar-algebras $C^*_\red(\cB_H)\sbe C^*_\red(\cB)$ we take a nondegenerate *-representation $T\colon \cB\to \bB(X)$ with $T|_{B_e}$ faithful.
By \cite[Corollary 2.15]{ExNg}, there exists a faithful *-representation $\pi\colon C^*_\red(\cB)\to \bB(L^2(G,X))$ such that, for all $f\in C_c(\cB),$ $\pi(f)=[\lambda^G\tilde{\otimes}T]_f.$
For all $f\in C_c(\cB_H)\sbe C_c(\cB),$ $\pi(f) = [\lambda^G|_H\tilde{\otimes}T|_{\cB_H}]_f$ and $\lambda^G|_H$ is a direct sum of copies of $\lambda^H$ because $G$ is the disjoint union of open cosets $Ht$.
Then $\|f\|_{C^*_\red(\cB)}=\|[\lambda^G\tilde{\otimes}T]_f\|=\|[\lambda^H\tilde{\otimes}T|_{\cB_H}]_f\|=\|f\|_{C^*_\red(\cB_H)},$ which implies that we may identify $C^*_\red(\cB_H)$ with the closure of $C_c(\cB_H)$ in $C^*_\red(\cB)$.

Let $P\in \bB(L^2(G,X))$ be the orthogonal projection onto $L^2(H,X).$
For all $f\in C_c(\cB)$ we have $P(\lambda^G\tilde{\otimes }T)_fP=P(\lambda^G\tilde{\otimes }T)_{f|_H}P$ and
\begin{equation*}
    \|f|_H\|_{C^*_\red(\cB_H)} =\|(\lambda^H\tilde{\otimes}T|_{\cB_H})_{f|_H}\|= \|P(\lambda^G\tilde{\otimes }T)_{f|_H}P\|\leq \|(\lambda^G\tilde{\otimes }T)_{f}\|=\|f\|_{C^*_\red(\cB)}.
\end{equation*}
Then the continuous extension $ E_\red\colon C^*_\red(\cB)\to C^*_\red(\cB_H)$ of $E_c$ is a conditional expectation.

\begin{remark}
    If $C^*_\red(\cB)$ is nuclear, then so it is $C^*_\red(\cB_H)$.
\end{remark}

\begin{remark}
    If $\cB$ has the \wcp, then $\cB_H$ has the \wcp\ because the closure of $C_c(\cB_H)$ in $C^*(\cB)$ is canonically isomorphic to the closure of $C_c(\cB_H)$ in $C^*_\red(\cB).$
\end{remark}

\subsection{The universal completion of the partial cross-sectional bundle}\label{ssec: full completion of partial cross}
Assume that $H$ is a closed and normal subgroup of $G$.
 Let $\cL=\{L^1(\cB_u)\}_{u\in G/H}$ be the partial cross-sectional bundle over $G/H$ derived from $\cB$, and $\qb{\cB}{H}=\{\fqb{B}{u}\}_{u\in G/H}$ the bundle \Cstar-completion of $\cL$, see \cite[VIII 6 \& VIII 16.7]{FlDr88}.
    By construction, each fiber $L^1(\cB_u)$ is the $\|\ \|_1-$completion of $C_c(\cB_u)$, where $\cB_u=\{B_t\}_{t\in u}$ is the reduction of $\cB$ to the coset $u\in G/H$.
    The natural map $C_c(\cB)\to C_c(\cL),$ $f\mapsto \tilde{f}$ with $\tilde{f}(u)=f|_u$, extends to a *-isomorphim $\Phi\colon L^1(\cB)\to L^1(\cL)$ (see \cite[VIII 6.7]{FlDr88}).
    The cross-sectional \Cstar-algebras of $\cB$ and $\cL$ are the envelping \Cstar-algebras of $L^1(\cB)$ and $L^1(\cL)$, so we may think $C^*(\cB)=C^*(\cL)$.
    In addition, the map $\rho\colon \cL\to \qb{\cB}{H}$ of \cite[VIII 17.7]{FlDr88} is a dense embedding in the sense of \cite[XI 12.6]{FlDr88}, so $\{\rho\circ f\colon f\in C_c(\cL)\}$ is dense in $C_c(\qb{\cB}{H})$ with respect to the inductive limit topology (use \cite[II 14.6]{FlDr88}).
    The construction of the bundle \Cstar-completion implies that all *-representations of $\cL$ have the form $T\circ\rho$ for a *-representation $T$ of $\qb{\cB}{H}.$
    Then the continuous *-homomorphism $L^1(\cL)\to L^1(\qb{\cB}{H})$ mapping $f\in C_c(\cL)$ to $\rho\circ f$ induces an isomorphism between $C^*(\cL)$ and $C^*(\qb{\cB}{H})$.
    This yields $C^*(\cB)=C^*(\qb{\cB}{H})$ or, more precisely, there exists an isomorphism 
    \begin{equation}\label{equ: iso full full}
     \Omega\colon C^*(\cB)\to C^*(\qb{\cB}{H})  \textrm{ such that }\forall f\in C_c(\cB),\ \Omega(f)=\rho \circ \tilde{f}.
    \end{equation}

    We claim that the fiber $\fqb{B}{H}$ of $\qb{\cB}{H}$ over the unit $H$ of $G/H$ is $C^*(\cB_H)$.
    For the bundle $\cL$, \cite[XI 11.6]{FlDr88} says that $C^*(\cB_H)=\fqb{B}{H}$ if and only if every *-representation of $L^1(\cB_H)$ is $\cL$-positive.
    Using \cite[XI 12.7]{FlDr88} we get that this last claim is equivalent to say that every *-representation of $\cB_H$ is $\cB$-positive, which is the case \cite[Theorem 2.1]{Ferraro_2024}. 

\subsection{The reduced completion of the partial cross-sectional bundle}\label{ssec: reduced completion of partial cross}
We want to construct a reduced version $\qbr{\cB}{H}$ of $\qb{\cB}{H}$ so that $C^*_\red(\cB)$ is isomophic to $C^*_\red(\qbr{\cB}{H})$ ($H$ is normal and closed in $G$).
To do this we take a nondegenerate *-representation $T\colon \cB\to \bB(X)$ with $\phi:=T|_{B_e}$ faithful, and let $S\colon \cL\to \bB(Y)$ be the *-representation corresponding to  $\Ind_{B_e\uparrow \cB}(\phi)$ via the process of \cite[VIII 15.9]{FlDr88}.
Using $S$ and the construction of \cite[VIII 16.7]{FlDr88} we can produce a Fell bundle $\qbr{\cB}{H}=\{\fqbr{B}{u}\}_{u\in G/H}$ and a dense embedding  $\rho^\red\colon \cL\to \qbr{\cB}{H}$ such that, for all $f\in \cL,$ $\|\rho^\red(f)\|=\|S_f \|$.
By construction, for all $r\in G$; $f\in C_c(\cB_{rH})$ and $\xi\in Y,$ $S_f\xi = \int_N \Ind_{B_e\uparrow\cB}(\phi)_{f(rt)}\xi dt$.
Notice that we may replace $f(rt)$ with $\tilde{f}(rN)(rt)$.
This ultimately implies that $\tilde{S}\circ \Phi|_{C_c(\cB)} = \intInd_{B_e\uparrow\cB}(\phi)|_{C_c(\cB)}$.
    
The fiber $\fqbr{B}{H}$ of $\qbr{\cB}{H}$ over the unit $H$ of $G/H$ is the completion of $L^1(\cB_H)$ with respect to the seminorm $f\mapsto \|S_f\|.$
    Hence, to prove $\fqbr{B}{H}=C^*_\red(\cB_H)$ it suffices to show that for all $f\in C_c(\cB_H)$ we have $\|f\|_{C^*_\red(\cB_H)}=\|S_f\|.$
    By Fell's absorption principle \cite{ExNg,Ferraro_2024}, $\Ind_{B_e\uparrow\cB}(\phi)\equiv \Ind_{B_e\uparrow\cB}(T|_{B_e})$ is weakly equivalent ($\sim$) to $\lambda^G\otimes T.$
     The process of restricting *-representations is continuous with respect to the regional topology \cite[VIII 21.20]{FlDr88} and, by definition, it is a homeomorphism with respect to the operation of taking integrated forms.
    This implies
    \begin{equation*}
     S|_{L^1(\cB_H)}\sim [(\lambda^G\otimes T)|_{\cB_H}]^{\tilde{\ }}=[(\lambda^G|_{H})\otimes (T|_{\cB_H})]^{\tilde{\ }}\sim \lambda^H\tilde{\otimes} (T|_{\cB_H})
    \end{equation*}
    where the last equivalence holds because  $\lambda^G|_H\sim \lambda^H$ and because the process of taking tensor products with representations of the base group is continuous with respect to the regional topology.
    By \cite[Corollary 2.15]{ExNg}, for all $f\in L^1(\cB_H)$ we have $\|S_f\|=\|\lambda^H\tilde{\otimes} (T|_{\cB_H}) (f) \|=\|f\|_{C^*_\red(\cB_H)}.$
    Then $\fqbr{B}{H} = C^*_\red(\cB_H)$.

    Now we want to prove that $C^*_\red(\cB)=C^*_\red(\qbr{\cB}{H}).$
    We know $\lambda^G\tilde{\otimes}T \sim \Ind_{B_e\uparrow\cB}(\phi),$ so there exists a faithful *-representation 
    \begin{align*}
        \mu&\colon C^*_\red(\cB)\to \overline{\intInd_{B_e\uparrow\cB}(\phi)(C_c(\cB))}
    \end{align*}
    such that $\mu(f)=\intInd_{B_e\uparrow\cB}(\phi)(f)=\tilde{S}(\tilde{f}),$ for all $f\in C_c(\cB).$
    
    By induction in stages, 
    \begin{equation*}
     \Ind_{B_e\uparrow \cB}(\phi) = \Ind_{\cB_H\uparrow \cB}(\Ind_{B_e\uparrow \cB_H}(T|_{B_e})).   
    \end{equation*}
    The integrated form of $\Ind_{B_e\uparrow \cB_H}(T|_{B_e})\sim \lambda^H\otimes (T|_{\cB_H})$ is a *-representation $R$ of $L^1(\cB_H)$ for which there exists a faithful *-representation $R'$ of $\fqbr{B}{H}=C^*_\red(\cB_H)$ such that $R'\circ \rho|_{L^1(\cB_H)}=R,$ see \cite[Corollary 2.15]{ExNg}.
    By \cite[XI 12.6-12.7]{FlDr88}, $S=\Ind_{L^1(\cB_H)\uparrow\cL}(R)$ and $\Ind_{\fqbr{B}{H}\uparrow\qbr{\cB}{H}}(R')\circ \rho^\red = \Ind_{L^1(\cB_H)\uparrow\cL}(R).$
    From Proposition 3.1 of \cite{Ferraro_2024} we get the existence of a faithful *-representation 
    \begin{equation*}
        \nu\colon C^*_\red(\qbr{\cB}{H})\to \overline{\intInd_{\fqbr{B}{H}\uparrow\qbr{\cB}{H}}(R)(C_c(\qbr{\cB}{H}))}
    \end{equation*}
    such that $\nu(f)=\intInd_{\fqbr{B}{H}\uparrow\qbr{\cB}{H}}(R')(f),$ for all $f\in C_c(\qbr{\cB}{H}).$

    For all $f\in C_c(\cB)$ we have 
    \begin{align*}
        \mu(f) =\tilde{S}(\tilde{f})
         = \intInd_{L^1(\cB_H)\uparrow\cL}(R)(\tilde{f})
         = \intInd_{\fqbr{B}{H}\uparrow\qbr{\cB}{H}}(R')(\rho^\red\circ \tilde{f})
         =\nu( \rho^\red\circ \tilde{f} ).
    \end{align*}
    The function $\Omega^\red\colon C_c(\cB)\to C_c(\qbr{\cB}{H}), $ $f\mapsto \rho^\red\circ \tilde{f},$ is a *-homomorphism because $\nu\circ \Omega^\red=\mu$.
    Using \cite[II 14.6]{FlDr88} one can show that the image of $\Omega^\red$ is dense with respect to the inductive limit topology, so $\Omega^\red(C_c(\cB))$ is dense in $C^*(\qbr{\cB}{H})$ and  $\Omega^\red$ extends to an isomorphism
    \begin{equation}\label{equ: iso red red}
        \Omega^\red\colon C^*_\red(\cB)\to C^*_\red(\qbr{\cB}{H}).
    \end{equation}

    \begin{notation}
        The fiber over $u\in G/H$ of $\qbr{\cB}{H}$ will be denoted $C^*_\red(\cB_u).$
    \end{notation}
    
\subsection{Weak containment, nuclearity and amenability}\label{sec:wcp-nuc-amen}
We go back to our original assumption of $H$ being closed in $G$ (it is not normal anymore, unless specifically stated).

\begin{theorem}
    If the normaliser $N$ of $H$ in $G$ is open (e.g. $G/H$ is discrete), then the following claims are true.
    \begin{enumerate}
        \item If $\cB$ has the \wcp, then $\cB_H$ has the \wcp.
        \item If $C^*(\cB)$ is nuclear, then $C^*(\cB_H)$ is nuclear.
        \item If $C^*_\red(\cB)$ is nuclear, then $C^*_\red(\cB_H)$ is nuclear.
    \end{enumerate}
\end{theorem}
\begin{proof}
    The first claim is \cite[Corollary 5.6]{Ferraro_2024}.
    Since $(\cB_N)_H = \cB_H$ and we have conditional expectations, $E\colon C^*(\cB)\to C^*(\cB_N)$ and $E_\red\colon C^*_\red(\cB)\to C^*_\red(\cB_N)$, we may assume that $N=G$ (i.e. $H$ is normal in $G$).
    We go a step further and use a trick of \cite{FlDr88} to reduce the proof to $H=\{e\}$.

    If $C^*(\cB)=C^*(\qb{\cB}{H})$ is nuclear, then $C^*_\red(\qb{\cB}{H})$ is nuclear and Theorem~\ref{thm: reduced algebra nuclear implies nuclear fiber} implies that the fiber $C^*(\cB_H)$ of $\qb{\cB}{H}$ over $H\in G/H$ is nuclear.
    Replacing $\qb{\cB}{H}$ with $\qbr{\cB}{H}$ and recalling that $C^*_\red(\cB)=C^*_\red(\qbr{\cB}{H})$, a similar argument implies that if $C^*_\red(\cB)$ is nuclear, then the fiber $C^*_\red(\cB_H)$ of $\qbr{\cB}{H}$ is nuclear.
\end{proof}

\begin{corollary}\label{cor: completion nuclear}
If $H$ is normal in $G$, $C^*(\cB_H)$ (resp. $C^*_\red(\cB_H)$) is nuclear and $\qb{\cB}{H}$ (resp. $\qbr{\cB}{H}$) is \Cstar-amenable, then $C^*(\cB)$ (resp. $C^*_\red(\cB)$) is nuclear.
\end{corollary}
\begin{proof}
    The fiber $C^*(\cB_H)$ of $\qb{\cB}{H}$ over the unit $H$ is nuclear and $\qb{\cB}{H}$ is \Cstar-amenable.
    By Theorem~\ref{thm: B amenable and Be nuclear} and \eqref{equ: iso full full}, $C^*(\cB)$ is nuclear.
    On the other hand, $C^*_\red(\cB_H)$ is the fiber of $\qbr{\cB}{H}$ over $H$ and, by \eqref{equ: iso red red}, $C^*_\red(\cB)$ is isomorphic to the nuclear algebra $C^*_\red(\qbr{\cB}{H})$.
\end{proof}

\begin{theorem}
    If $H$ is normal in $G$, then the following claims are equivalent
    \begin{enumerate}
        \item $\cB$ has the \wcp.
        \item $\qb{\cB}{H}$ and $\cB_H$ have the \wcp.
        \item $\qbr{\cB}{H}$ and $\cB_H$ have the \wcp.
    \end{enumerate}
\end{theorem}
\begin{proof}
    The first condition implies that $\cB_H$ has the \wcp, so we may continue the proof assuming that $\cB_H$ has the \wcp.
    In this case, there is no difference in the processes of constructing $\qb{\cB}{H}$ and $\qbr{\cB}{H}$ out of $\cL$ because, for all $f\in \cL$, we have $\|S_f\|^2= \|S_{f^**f}\|=\|f^**f\|_{C^*_\red(\cB)} = \|f^**f\|_{C^*(\cB_H)}=\|\rho(f)\|^2.$
    Hence, $\qb{\cB}{H}=\qbr{\cB}{H}$ and the last two claims of the statement are equivalent.
    
    Restricting the isomorphisms $\Omega\colon C^*(\cB)\to C^*(\qb{\cB}{H})$ and $\Omega^\red\colon C^*_\red(\cB)\to C^*_\red(\qbr{\cB}{H})$ of \eqref{equ: iso full full} and \eqref{equ: iso red red} to $C_c(\cB)$ we get  $f\mapsto \rho\circ \tilde{f}.$
    Because $\qb{\cB}{H}=\qbr{\cB}{H}$, we have a commutative diagram
    \begin{equation*}
        \xymatrix{ C^*(\cB) \ar[r]^\Omega \ar[d]_{\lambda^\cB} & C^*(\qb{\cB}{H})\ar[d]^{\lambda^{\qb{\cB}{H}}}\\C^*_\red(\cB)\ar[r]_{\Omega^\red} & C^*_\red(\qb{\cB}{H}) }.
    \end{equation*}
    and it becomes clear that $\cB$ has the \wcp\ if and only if $\qb{\cB}{H}$ has the \wcp.   
\end{proof}

The preceding results naturally lead to the question of whether, under the hypotheses of the theorem above, the following three conditions are equivalent:
\begin{enumerate}
    \item\label{item: B amenable in the thm for reductions and quotients} $\cB$ is \Cstar-amenable.
    \item\label{item: BH and quotient amenable} $\cB_H$ and $\qb{\cB}{H}$ are \Cstar-amenable.
    \item\label{item: BH and reduced quotient amenable} $\cB_H$ and $\qbr{\cB}{H}$ are \Cstar-amenable.
\end{enumerate}

The equivalence \eqref{item: BH and quotient amenable} $\iff$ \eqref{item: BH and reduced quotient amenable} is immediate: if $\cB_H$ is \Cstar-amenable, it has the \wcp, which yields the canonical identification $\qb{\cB}{H}=\qbr{\cB}{H}$. We establish the forward implication \eqref{item: B amenable in the thm for reductions and quotients} $\Rightarrow$ \eqref{item: BH and quotient amenable} in Theorem~\ref{thm: permanence quotient bundles} below. 

The reverse implication \eqref{item: BH and quotient amenable} $\Rightarrow$ \eqref{item: B amenable in the thm for reductions and quotients} remains an intriguing open question in full generality. We strongly suspect it holds. As evidence, we note that it is true under additional hypotheses (for example, if $G$ is inner amenable and $B_e$ is nuclear, which follows from applying Corollary~\ref{cor: completion nuclear} twice and Theorem~\ref{thm: reduced algebra nuclear implies nuclear fiber}). Furthermore, we have verified that the converse holds unconditionally when $G$ is discrete; the proof involves distinct techniques and will appear in a forthcoming article.

\begin{theorem}[Permanence for Quotient Bundles]\label{thm: permanence quotient bundles}
    If the Fell bundle $\cB$ over $G$ is \Cstar-amenable and $H$ is a closed normal subgroup of $G$, then the reduction $\cB_H$ and the partial quotient bundles $\qb{\cB}{H}=\qbr{\cB}{H}$ are \Cstar-amenable. 
\end{theorem}
\begin{proof}
    Theorem~\ref{thm: amenability and reductions} implies that $\cB_H$ is \Cstar-amenable.

    Forget about amenability for a moment and take a ccp map $\Phi\colon \cB\to \cB.$
    We claim that there exists a unique ccp map $\qb{\Phi}{H}\colon \qb{\cB}{H}\to \qb{\cB}{H}$ such that for all $u\in G/H$ and $f\in C_c(\cB_u)\sbe L^1(\cB_u)$, $\qb{\Phi}{H}\circ \rho (f)=\rho(\Phi\circ f)$.

    Take a nondegenerate *-representation $T\colon C^*(\cB)\to \bB(X)$ with faithful integrated form.
    Via the isomorphism $L^1(\cL)\cong L^1(\cB)$, $T$ induces a *-representation $T^1\colon \cL\to \bB(X)$, see \cite[VIII 15.9]{FlDr88}, and $T^1$ determines a unique *-representation $\dot{T}\colon \qb{\cB}{H}\to \bB(X)$ such that $\dot{T}\circ \rho =T^1$.
    Moreover, the integrated forms of $\dot{T}$ and $T$ are the same up to the isomorphism of  \eqref{equ: iso full full}.
    Since the integrated form of $\dot{T}$ is faithful, the restrictions of $\dot{T}$ to each one of the fibers is isometric.
    Hence, for all $f\in \uplus_{u\in G/H}C_c(\cB_u)$, $\| \rho(\Phi\circ f) \|= \|\dot{T}_{\Phi\circ f}\|$.

    Take a nondegenerate dilation of $T\circ \Phi$, $S\colon \cB\to \bB(Y)$, and $V\in \bB(X,Y)$ such that $T\circ \Phi(b)=V^* S_bV$ and $\|V\|\leq 1$.
    Using the explicit descriptions $T^1$ and $S^1$ of \cite[VIII 15.9]{FlDr88}, a straightforward computation reveals that for all $f\in \uplus_{u\in G/H}C_c(\cB_u)$, $ T^1_{\Phi\circ f} = V^*S^1_f V = V^* \dot{S}_{\rho(f)}V$ and it follows that $\|\rho(\Phi\circ f)\|\leq \|\rho(f)\|$.
    Since $\rho(C_c(\cB))$ is dense in $\fqb{B}{u}$, there exists a unique family of linear contractions $\{ \fqb{\Phi}{u}\colon \fqb{B}{u}\to \fqb{B}{u} \}_{u\in G/H}$ whose union $\qb{\Phi}{H}\colon \qb{\cB}{H}\to \qb{\cB}{H}$ satisfies $\qb{\Phi}{H}\circ \rho(f)=\rho(\Phi\circ f)$, for all $f\in \uplus_{u\in G/H}C_c(\cB_u)$.
    For all $f\in C_c(\cB)$ we have $\qb{\Phi}{H}\circ (\rho\circ \tilde{f})(u) =  \qb{\Phi}{H}\circ (\rho(f|_u)) = \rho((\Phi\circ f)|_u) = \widetilde{\Phi\circ f}(u)$.
    Considering the family $\Gamma:=\{\rho\circ \tilde{f}\colon f\in C_c(\cB)\}$ and using \cite[II 13.16]{FlDr88} we get that $\qb{\Phi}{H}$ is continuous and it is ccp because $T\circ \qb{\Phi}{H} (z) = V^* \dot{S}_z V$.
    Notice that $\supp(\qb{\Phi}{H})$ is contained in the projection of $\supp(\Phi)$ into $G/H$.

    Let $\{\Phi_i\colon \cB\to \cB\}_{i\in I}$ be a net of ccp maps with compact support converging uniformly to $\id_\cB$ on compact slices.
    Then, $\{\qb{\Phi_i}{H}\colon\qb{\cB}{H}\to \qb{\cB}{H}\}_{i\in I}$ is a net of ccp maps with compact supports.
    To complete the proof we use Lemma~\ref{lemma: convergence on gamma suffices} and the family $\Gamma$ defined above in this proof.
    For all $f\in C_c(\cB)$ and $t\in G$, setting $g:=\rho\circ \tilde{f}$, we get   
    \[
    \|  \qb{\Phi_i}{H}\circ g(tH)-g(tH)\|
        \leq \| T^1_{\Phi_i\circ (f|_{tH}) - f|_{tH} } \|\leq \| \Phi_i\circ (f|_{tH}) - f|_{tH} \|_1.
        \]
    and this yields
    \[ \|  \qb{\Phi_i}{H}\circ g-g\|_1 
        \leq \int_{G/H} \| \qb{\Phi_i}{H}\circ g(tH)-g(tH)\|
         dtH \leq \int_{G/H} \int_H \| \Phi_i(f(ts))-f(ts)\|ds dtH =\|\Phi_i\circ f-f\|_1.  \]
    Thus, $\lim_i \|  \qb{\Phi_i}{H}\circ g-g\|_1 = 0$.
\end{proof}

\appendix 

\section{Tensor products of equivariant cp maps between \Wstar-algebras}\label{sec: appendix tensor product of cp equivariant maps}

Let $\gamma$ and $\delta$ be \Wstar-actions of a group $G$ on $M$ and $N,$ respectively, and $\pi\colon M\to N$ an equivariant ccp map.
Fix unitary implementations $(M\sbe \bB(X),U)$ and $(N\sbe \bB(Y),V)$ and a Hilbert space $Z$.

\begin{remark}[The coefficient function]\label{rmk:ternary operation of coefficients}
    The function $Z\times Z\times \bB(Z\otimes X)\to \bB(X)$ given by $(z,w,T)\mapsto T_{z,w},$ where $\langle T_{z,w}u,v\rangle = \langle T(w\otimes u),z\otimes v\rangle$, is conjugate linear the first variable and linear in the rest.
    Moreover, it is a contraction in the sense that $\|T_{z,w}\|\leq\|z\|\|w\|\|T\|$.
    Hence, $(z,w,T)\mapsto T_{z,w}$ is continuous with respect to the product of the norm topologies.
\end{remark}

Recall that given  any Hilbert base $S$ of $Z$ we have
\[\bB(Z)\bo M=\{T\in \bB(Z\otimes X)\colon \forall \ z,w\in S,\ T_{z,w}\in M\}.\]
Thinking of each $T\in \bB(Z\otimes X)$ as the matrix $(T_{z,w})_{z,w\in S}$ one can construct a function
\[\id \bo_S\pi\colon \bB(Z)\bo M\to \bB(Z)\bo N\]
that in matrix form can be expressed as
\begin{equation}\label{equ: explicit formula por 1 barotimes pi}
  (\id \bo_S\pi)((T_{z,w})_{z,w\in S})=(\pi(T_{z,w}))_{z,w\in S}.  
\end{equation}
In particular, $(\id \bo_S\pi)(T\otimes m)=T\otimes \pi(m)$ for all $T\in \bB(Z)$ and $m\in M$.

To justify the existence of $\id \bo_S\pi$ one can proceed as follows.
For each finite $F\sbe S$ let $P_F\in \bB(Z)$ be the orthogonal projection onto $\spn(F)$.
Since $P_F \bB(Z)P_F$ is a finite dimensional \Cstar-algebra, 
\[(P_F\otimes 1_X)\bB(Z)\bo M(P_F\otimes 1_X)=P_F\bB(Z)P_F\otmin M\]
and there exists a unique ccp map $\id\otimes \pi\colon P_F\bB(Z)P_F\otmin M\to \bB(Z)\bo N$ sending $T\otimes m$ to $T\otimes \pi(m)$.
For all $T\in \bB(Z)\bo M,$ $\|(\id\otimes \pi)( (P_F\otimes 1_X)T(P_F\otimes 1_X) )\|\leq \|T\|$.
Taking $F$ in the directed set of finite subsets of $S$, some subnet of $\{(\id\otimes \pi)( (P_F\otimes 1_X)T(P_F\otimes 1_X) )\}_F$ has a limit in the weak operator topology (\wot).
Expressed in matrix form, that limit must be $(\pi(T_{z,w}))_{z,w\in S}$.
This implies that $\{(\id\otimes \pi)( (P_F\otimes 1_X)T(P_F\otimes 1_X) )\}_F$ $\wot$-converges to $(\pi(T_{z,w}))_{z,w\in S}$ (because all of it subnets have a subnet converging to that operator).
Then we can define 
\[  (\id \bo_S \pi)(T):=\wot\lim_F (\id\otimes \pi)( (P_F\otimes 1_X)T(P_F\otimes 1_X) )\]
and it follows that $(\id \bo_S \pi)$ is ccp.
Thinking in matrix form, it follows that $\pi$ is normal if and only if $\id \bo_S\pi$ is normal.
In addition, $\id\bo_S\pi=\id$  (the identity map of the domain of $\id\bo_S\pi$) if and only if $\pi=\id$.

\begin{lemma}
    If $S_1$ and $S_2$ are Hilbert bases of $Z$, then $\id \bo_{S_1} \pi=\id \bo_{S_2} \pi$.
\end{lemma}
\begin{proof}
    It suffices to show that for all $T\in \bB(Z)\bo M$ and $z,w\in S_1$ we have $(\id \bo_{S_2}\pi)(T)_{z,w} = (\id \bo_{S_1}\pi)(T)_{z,w}.$
    Taking limit over the finite sets $F$ of $S_2$ and using Remark~\ref{rmk:ternary operation of coefficients} we get
    \begin{align*}
        (\id \bo_{S_2}\pi)(T)_{z,w}
        & = \lim_F (\id \bo_{S_2}\pi)(T)_{P_Fz,P_Fw}
        = \lim_F \sum_{x,y\in F} \langle z,x\rangle \langle y,w\rangle (\id \bo_{S_2}\pi)(T)_{x,y}\\
        & = \lim_F \sum_{x,y\in F} \langle z,x\rangle \langle y,w\rangle \pi(T_{x,y})
        = \lim_F \pi(T_{P_F z,P_Fw}) = \pi(T_{z,w})=(\id \bo_{S_1}\pi)(T)_{z,w},
    \end{align*}
    which yields the desired result.
\end{proof}

From now on we can safely write $\id \bo \pi$ instead of $\id \bo_S \pi.$
The tensor products $\bB(Z)\bo M$ and $\bB(Z)\bo N$ have canonical structures of $\bB(Z)-$bimodules, it turns out that $\id \bo \pi$ is a bimodule map:

\begin{lemma}\label{lemma: bimodule property}
    For all $R,S\in \bB(Z)$ and $T\in \bB(Z)\bo M,$
    \[(\id \bo\pi)( (R\otimes 1_X)T(S\otimes 1_X) )=(R\otimes 1_Y)(\id \bo \pi)(T)(R\otimes 1_Y).\]
\end{lemma}
\begin{proof}
    Since $\id \bo \pi$ preserves the involution, it suffices to consider the case $R=1_X$.
    Moreover, using the norm continuity of $\id \bo \pi$ and Borel Functional Calculus, we can further reduce the proof to the case where $S$ is a projection.

    Let $Z'$ be the union of Hilbert bases of $SZ$ and $(SZ)^\perp$.
    Take $z,w\in Z'$ and $u,v\in Y$.
    We have
    \begin{align*}
        \langle (\id \bo \pi)(T)(S\otimes 1_Y)(z\otimes u),w\otimes v\rangle 
        & = \langle Sz,z\rangle \langle (\id \bo \pi)(T)(z\otimes u),w\otimes v\rangle \\
        & = \langle Sz,z\rangle \langle \pi(T_{w,z})u,v\rangle
        = \langle Sz,z\rangle \langle \pi(T_{w,z})u,v\rangle.
    \end{align*}
    Once again, replacing $\pi$ with the identity of $M$ we get that $[T(S\otimes 1_X)]_{w,z}=\langle Sz,z\rangle T_{w,z}$ and we obtain
    \[ \langle (\id \bo \pi)(T)(S\otimes 1_Y)(z\otimes u),w\otimes v\rangle =\langle (\id \bo \pi)(T(S\otimes 1_X))(z\otimes u),w\otimes v\rangle;  \]
    which implies $(\id \bo \pi)(T)(S\otimes 1_Y)=(\id \bo \pi)(T(S\otimes 1_X))$.
\end{proof}

\begin{corollary}\label{cor: one tensor pi and tensor product}
    If $P\sbe \bB(Z)$ is a \Wstar-subalgebra, then $(\id \bo \pi)(P\bo M)\sbe P\bo N.$
\end{corollary}
\begin{proof}
    Assume that $1_X\in P$.   
    We know that $(\id \bo \pi)(P\bo M)\sbe \bB(Z)\bo N$.
    Combining Lemma~\ref{lemma: bimodule property} with \cite[Corollary 5.9]{takesaki1979theory} we obtain $(\id \bo \pi)(P\bo M)\sbe (P'\bo \bC 1_Y)'=P\bo \bB(Y)$.
    Now that we know
    \[(\id \bo \pi)(P\bo M)\sbe (P\bo \bB(Y))\cap (\bB(Z)\bo N),\]
    we can use  \cite[Corollary 5.10]{takesaki1979theory} to get $(\id \bo \pi)(P\bo M)\sbe P\bo N$.

    Now assume that the unit $p$ of $P$ is not $1_Z$ and let $P\to \tilde{P}\sbe \bB(pZ)$ be given by $m\mapsto \tilde{m}:=m|_{pZ}$.
    The decomposition $Z= (1_Z-p)Z \oplus pZ$ induces $P' = \bB((1_Z-p)Z)\oplus \tilde{P}'$ and $P'' = \bC (1_Z - p)\oplus \tilde{P}$ is a \Wstar-subalgebra of $\bB(Z)$ containing $1_X$.
    Hence, $(\id \bo \pi)(P''\bo M)\sbe P''\bo N.$
    We know $P=0\oplus \tilde{P}$ and $p=0\oplus \tilde{p}.$
    Then,
    \begin{align*}
      (\id \bo \pi)(P\bo M)
      = (\id \bo \pi)((p\otimes 1_X)P''(p\otimes 1_X)) 
      = (p\otimes 1_Y)(P''\bo N)(p\otimes 1_X) = (0\oplus \tilde{P})\bo M = P\bo M,
    \end{align*}
    which completes the proof in the general case.
\end{proof}

Notice that 
\begin{align*}
    \id\bo\gamma& \colon G\times (\bB(Z)\bo M)\to \bB(Z)\bo M & (t,T)&\mapsto (\id\bo \gamma_t)(T)= (1\otimes U_t)T(1\otimes U_t)^*
\end{align*}
is a \Wstar-action of $G$ on $\bB(Z)\bo M$ with unitary implementation $(\bB(Z)\bo M\sbe \bB(Z\otimes X), 1\otimes U).$

\begin{lemma}\label{lemma: equivariance}
    The map $\id \bo \pi$ is equivariant with respect to $\id\bo \gamma$ and $\id\bo \delta$.
\end{lemma}
\begin{proof}
Fortunately, $\id\bo\gamma$ and $\id\bo\delta$ are implemented by the unitary representations $t\mapsto 1\otimes U_t$ and $t\mapsto 1\otimes V_t$.
Using \eqref{equ: explicit formula por 1 barotimes pi} we get that for all $T\in \bB(Z)\bo M$ and $t\in G$
\begin{align*}
   (\id \bo \pi)( (1\otimes U_t)T(1\otimes U_t)^* ) 
   &= (\id \bo\pi)( (U_tT_{z,w}U_t^*)_{z,w\in S})
   = (\pi (U_tT_{z,w}U_t^*)_{z,w\in S}) \\
   & = (V_t\pi (T_{z,w})V_t^*)_{z,w\in S} 
   =(1\otimes V_t)\pi(T)(1\otimes V_t)^*,
\end{align*}
so $\id \bo_S\pi \equiv \id \bo \pi$ is equivariant.
\end{proof}

All we have done so far can be repeated flipping the tensors, we mean using $M\bo \bB(Z)$ instead of $\bB(Z)\bo M$, so we get a ccp map $\pi\bo \id\colon M\bo\bB(Z)\to N\bo\bB(Z)$.
At this point we have all the necessary ingredients to prove equivariant versions of some results by Nagisa and Tomiyama \cite[Theorem 3]{NagTomCP}, \cite[Theorem 4]{Tomiyama_tensor_vN_algebras}.

\begin{theorem}\label{thm: equivariant vN tensor of equivariant maps}
    Let (for $i=1,2$) $G_i$ be a group and $\pi_i\colon M_i\to N_i$ be a ccp equivariant map with respect to the \Wstar-actions $\gamma^i$ and $\delta^i$ on $M_i$ and $N_i$, respectively.
    Then
    \begin{enumerate}
        \item The (usual) von Neumann tensor product action $\gamma^1\bo\gamma^2$ of $G_1\times G_2$ on $M_1\bo M_2$ is
    \begin{align*}
     G_1\times G_2&\times (M_1\bo M_2)\to M_1\bo M_2 & (t_1,t_2,T)&\mapsto (\gamma^1_{t_1}\bo\gamma^2_{t_2})(T):=(\gamma^1_{t_1}\bo \id)( (\id\bo\gamma^2_{t_2}) (T) ).
    \end{align*}
    \item There exists a ccp map $\rho\colon M_1 \bo M_2\to N_1 \bo N_2$ which is $\gamma^1\bo\gamma^2-\delta^1\bo\delta^2$ equivariant and for all $m_1\in M_1$ and $m_2\in M_2$,  $\rho(m_1\otimes m_2)=\pi_1(m_1)\otimes \pi_2(m_2)$.
    \item If $\pi_1$ and $\pi_2$ are normal, then $\rho$ can be made normal.
    \item If $\pi_1$ and $\pi_2$ are (normal) conditional expectations onto \Wstar-subalgebras, then $\rho$ can be made a (normal) conditional expectation.
    \end{enumerate}
\end{theorem}
\begin{proof}
     Take unitary implementations $(M_i\sbe \bB(X_i),U^i)$ and $(N_i\sbe \bB(Y_i),V^i)$.
     By Corollary~\ref{cor: one tensor pi and tensor product}, $(\id \bo \pi_2)(M_1 \bo M_2)\sbe M_1 \bo N_2$ and $(\pi_2\bo 1)(M_1 \bo N_2)\sbe N_1 \bo N_2.$
    Then $\rho\colon M_1 \bo M_2\to N_1 \bo N_2,$ $T\mapsto (\pi_2\bo 1)((\id \bo \pi_2)(T)),$ is ccp.
    In addtion, for all $m_i\in M_i$, $\rho(m_1\otimes m_2)=\pi_1(m_1)\otimes \pi_2(m_2)$.
    Regarding the normality claims, recall that $\id\bo\pi_2$ (respectively, $\pi_1\bo 1$) is normal if and only if $\pi_2$ ($\pi_1$) is normal, so $\rho$ is normal if both $\pi_1$ and $\pi_2$ are normal

    After recalling that each $\gamma^i_t$ is a $\wstar$-homeomorphism (because they are \Cstar-isomorphism) one gets that for all $(t_1,t_2,T)\in G_1\times G_2\times M_2\bo M_2$, $(\gamma^1_{t_1}\bo \id)( (\id\bo\gamma^2_{t_2}) (T) ) = (U^1_{t_1}\otimes U^2_{t_2})T(U^1_{t_1}\otimes U^2_{t_2})^*.$
    Then $(t_1,t_2,T)\mapsto (\gamma^1_{t_1}\bo \id)( (\id\bo\gamma^2_{t_2}) (T) )$ is the action with unitary implementation $(M_1\bo M_2\sbe \bB(X_1\otimes X_2),U^1\otimes U^2),$ i.e. the usual action of $G_1\times G_2$ on $M_1\bo M_2$.

    To prove $\rho$ is equivariant we take $T\in M_1 \bo M_2$, $t_1\in G_1$ and $t_2\in G_2$.
    Using Lemma~\ref{lemma: equivariance} and Corollary~\ref{cor: one tensor pi and tensor product}
    we get
    \begin{align*}
        \rho( (\gamma^1_{t_1}\bo \gamma^2_{t_2})(T))
        &= (\pi_2\bo 1)( (\id \bo \pi_2) (  (U^1_{t_1}\otimes U^2_{t_2})T (U^1_{t_1}\otimes U^2_{t_2})^*) )\\
        &= (\pi_2\bo 1)( (\id \bo \pi_2) ( (U^1_{t_1}\otimes 1) (1\otimes U^2_{t_2})T (1\otimes U^2_{t_2})^*(U^1_{t_1}\otimes 1)^*) )\\
        &= (\pi_2\bo 1)((U^1_{t_1}\otimes 1) (\id \bo \pi_2) (  (1\otimes U^2_{t_2})T (1\otimes U^2_{t_2})^*) (U^1_{t_1}\otimes 1)^*)\\
        &=(V^1_{t_1}\otimes 1) (\pi_2\bo 1)( (1\otimes V^2_{t_2})(\id \bo \pi_2) (  T  )(1\otimes V^2_{t_2})^*)(V^1_{t_1}\otimes 1)^*\\
        &=(V^1_{t_1}\otimes 1)(1\otimes V^2_{t_2}) (\pi_2\bo 1)( (\id \bo \pi_2) (  T  ))(1\otimes V^2_{t_2})^*(V^1_{t_1}\otimes 1)^*\\
        & =  (\delta^1_{t_1}\bo \delta^2_{t_2})(\rho(T)),
    \end{align*}
    so $\rho$ is equivariant.

    In case $\pi_1$ and $\pi_2$ are conditional expectations, let $p_i$ be the unit of $N_i$.
    Then we may take $Y_i=p_i X_i$ and use the implementations $(N_i\sbe \bB(pX_i),t\mapsto pU^i_tp)$.
    We are implicitly considering two faithful normal representations of $N_1\bo N_2,$ one in $\bB(p_1X_1\bo p_2X_2)$ and one in $\bB(X_1\otimes X_2)$.
    Formally, \eqref{equ: explicit formula por 1 barotimes pi} implies that for every  $T\in N_1 \bo N_2\sbe \bB(X_1\otimes X_2)$, $\rho(T) = (\pi_1 \bo 1) ( \id \bo \pi_2(T) )=(\pi_1 \bo 1) (T)=T\in \bB(p_1X_1\bo p_2X_2)$.
    Then $\rho$ is a conditional expectation. 
\end{proof}

Lemma 2.1 of \cite{ADaction1979} follows from the Theorem above by taking $G_1=G_2$ and identifying $G_1\cong \{(t,t)\colon t\in G_1\}\sbe G_1\times G_1$.
After this one gets \cite[Proposition 3.9]{ADaction1979}, which can be stated as

\begin{corollary}
    If, for $i=1,2$, $\gamma^i$ is a \Wstar-action of $G$ on  $M_i$ and at least one of them is \Wstar-amenable, then the diagonal action $\gamma^1 \bo^d \gamma^2:=(\gamma^1 \bo \gamma^2)|_{G}$ of $G$ on $M_1 \bo M_2$ is \Wstar-amenable.
\end{corollary}

The ``non-diagonal'' version would be

\begin{corollary}
    If, for $i=1,2$, $\gamma^i$ is a \Wstar-amenable action of $G_i$ on $M_i$, then action  $\gamma^1 \bo \gamma^2$ of $G_1\times G_2$ on $M_1 \bo M_2$ is \Wstar-amenable.
\end{corollary}
\begin{proof}
    Use Theorem~\ref{thm: equivariant vN tensor of equivariant maps} and the canonical isomorphism
    \[ (\li{G_1}\bo M_1)\bo (\li{G_2}\bo M_2)\cong \li{G_1\times G_2}\bo (M_1 \bo M_2),\]
    the details are left to the reader.
\end{proof}

\section{An extension Lemma for the continuous part of a \Wstar-action}\label{sec: an extension lemma}

Let $\gamma$ be a \Wstar-action of $G$ on $M$ and $(M\sbe \bB(X),U)$ a unitary implementation.
We briefly recall some facts about the continuous part of $\gamma$.

As in \cite{ADactionII1982} we use the bilinear function $\lone{G}\times M\to M,\ (f,m)\mapsto \gamma_f(m),$ such that for all $\varphi\in M_*,\ f\in \lone{G}$ and $m\in M$, $\varphi(\gamma_f(m))=\int_G \varphi(\gamma_t(m))f(t)dt$.
One way of proving the existence of such function is by using Riesz's representation Theorem to show that for each $(f,m)\in L^1(G)\times M$ there exists a unique operator $I^U_f(m)\in \bB(X)$ such that $\langle x,I^U_f(m)y\rangle = \int_G \langle U_t^*x, mU_t^* y\rangle f(t)dt=\int_G \langle x, \gamma_t(m) y\rangle f(t)dt$.
Since $\|\langle x, \gamma_t(m) y\rangle\|\leq \|m\|\|x\|\|y\|,$ we have $\|I^U_f(m)\|\leq \|m\|\|f\|_1$.
If $\omega$ is an ultra weak continuous functional of $\bB(X)$ such that $\omega(M)=\{0\},$ then $\omega(T)=\sum_{n\in \bN} \langle x_n,Ty_n\rangle$  for sequences $\{x_n\}_{n\in \bN},\{y_n\}_{n\in \bN}\subset X$ such that $\sum_{n\in \bN} \|x_n\|^2+\|y_n\|^2<\infty$; and $\omega(I^U_f(m))=\int_G \omega(\gamma_t(m))f(t)dt =0$ and it follows that $I^U_f(m)\in M$.
So we can set $\varphi_f(m):=I^U_f(m)$.

For $f,g\in \lone{G}$ and $s\in G$ we define $sf,fs\in \lone{G}$ by $sf(t)=f(\smu t)$ and $fs(t)=f(t\smu )\Delta(s)^{-1}$; while $f*g$ is the convolution product. 
With this notation, for all $m\in M$ one has
\begin{align*}
    \|\gamma_f(m)\| & \leq \|m\|\|f\|_1 & \gamma_s(\gamma_f(m))=&\gamma_{sf}(m) & \gamma_f(\gamma_s(m))& =\gamma_{fs}(m) & \gamma_f(\gamma_g(m))&=\gamma_{f*g}(m).
\end{align*}
In particular, $\|\gamma_r(\gamma_f(m))-\gamma_s(\gamma_f(m))\|= \|\gamma_{rf-sf}(m)\|\leq \|m\|\|rf-sf\|_1.$
For $f\in C_c(G),$ $r\mapsto rf$, is $\|\ \|_1$-continuous and it follows that $\{\gamma_f(m)\colon f\in C_c(G),\ m\in M\}\sbe M^c.$ 
Using an approximate unit of $\lone{G}$ contained in $C_c(G)$ one can show that $M^c$ is $\wstar$-dense in $M.$
For each $f\in C_c(G)^+,$ $\gamma_f\colon M\to M^c$ is completely positive because for all $m_1,\ldots,m_n\in M$ and $x_1,\ldots,x_n\in X$ we have $\sum_{i,j=1}^n \langle x_i,\gamma_f(m_i^*m_j)x_j\rangle = \int_G \|\sum_{i=1}^n m_iU_t^*x_i\|^2f(t)dt \geq 0$.

The idea of the proof below comes from that of \cite[Lemme 2.1]{ADactionII1982}.

\begin{lemma}\label{lemma: extension of pic}
If $(Y,\rho_c,V)$ is a unital covariant representation of $\gamma^c$, then there exists an equivariant ucp extension $\rho\colon M\to \bB(X)$ of $\rho_c$ with $\rho(M)'=\rho_c(M^c)'$.
\end{lemma}
\begin{proof}
    Let $I$ be the set of compact neighborhoods of $e\in G$ with the order $i\leq j\Leftrightarrow j\sbe i$.
    Fix a net $\{f_i\}_{i\in I}\sbe C_c(G)$ such that $f_i\geq 0$; $\int_G f_i(t)dt =1$  and $\supp(f_i)\sbe i$.
    The net $\{ \rho_c\circ \gamma_{f_i} \}_{i\in I}$ is contained in the set of ucp maps from $M$ to $\bB(X)$ and this set is compact with the topology of pointwise $\wstar$-convergence. 
    Hence, there exists a subnet $\{i_j\}_{j\in J}$ and a ucp map $\rho\colon M\to \bB(X)$ such that for all $m\in M$, $\rho(m)=\wstar \lim_j \rho_c(\gamma_{f_{i_j}}(m))$.
    For all $m\in M^c$ and $x,y\in X$ we have $\gamma_{f_i}(m)=\int_G \gamma_t(m)f_i(t)dt,$ because $t\mapsto \gamma_t(m)f_i(t)$ is norm-continuous and has compact support.
    This implies
    \[\langle \rho(m)x,y\rangle = \lim_j \int_G \langle \rho_c(m)U_t^*x,U_t^*y\rangle f_{i_j}(t)dt = \langle \rho_c(m)U_e^*x,U_e^*y\rangle = \langle \rho_c(m)x,y\rangle;\]
    and it follows that $\rho$ is an extension of $\rho_c$.

   To prove that $\rho$ is equivariant we take $m\in M$, $t\in G$ and $x,y\in X$.
   We have $\langle \rho(\gamma_t(m))x,y\rangle = \lim_j \langle \rho_c( \gamma_{f_{i_j}}(\gamma_t(m)))x,y\rangle=\lim_j \langle \rho_c( \gamma_{f_{i_j}t}(m))x,y\rangle$.
   The net $\{ f_{i_j}t * f_{i_k} \}_{k\in J}$ converges to $f_{i_j}t$ in $\lone{G}$, so $\gamma_{f_{i_j}t*f_{i_k}}(m)=\gamma_{f_{i_j}t}(\gamma_{f_{i_k}}(m))$ converges to $\gamma_{f_{i_j}t}(m)\in M^c$ in norm.
   Besides, $s\mapsto \gamma_s( \gamma_{f_{i_k}}(m))f_{i_j}t(s)$ is norm continuous and has compact support, so it is integrable and its norm-integral is $\gamma_{f_{i_j}t}(\gamma_{f_{i_k}}(m)).$
   Combining these facts we get
   \begin{align*}
        \langle \rho(\gamma_t(m))x,y\rangle & = \lim_j\lim_k \langle \rho_c( \gamma_{f_{i_j}t}(\gamma_{f_{i_k}}(m)))x,y\rangle 
   =\lim_j\lim_k \langle \rho_c( \int_G \gamma_s(\gamma_{f_{i_k}}(m))f_{i_j}t(s)ds)x,y\rangle\\
   & = \lim_j\lim_k \int_G\langle U_s \rho_c( \gamma_{f_{i_k}}(m)))U_s^* x,y\rangle f_{i_j}(s\tmu)\Delta(t)^{-1}ds\\
   & = \lim_j\lim_k \int_G\langle U_s U_t \rho_c( \gamma_{f_{i_k}}(m)))U_t^* U_s^* x,y\rangle f_{i_j}(s)ds.
   \end{align*}
  Notice that $\{U_t \rho_c( \gamma_{f_{i_k}}(m)))U_t^*\}_{k\in K}$ is a bounded net with \wot-limit $U_t\rho(m)U_t^*$.
  Besides, for all $j\in J$, the support of $f_{i_j}$ is compact and $s\mapsto U_s^*x$ is continuous.
  Combining these facts we can prove that the net $\{s\mapsto \langle U_{st} \rho_c( \gamma_{f_{i_k}}(m)))U_{st}^*x,y\rangle f_{i_j}(s)\}_{k\in K}\subset C_c(G)$ converges uniformly to $s\mapsto \langle U_{st}\rho(m)U_{st}^* x,y\rangle f_{i_j}(s)$.
  Thus, $\rho$ is equivariant because
  \begin{equation*}
        \langle \rho(\gamma_t(m))x,y\rangle  = \lim_j \int_G\langle U_s U_t \rho(m)U_t^* U_s^* x,y\rangle f_{i_j}(s)ds = \langle U_t \rho(m)U_t^* x,y\rangle.
  \end{equation*}

    Clearly, $\rho_c(M^c)\sbe \rho(M),$ so $\rho(M)'\sbe \rho_c(M^c)'.$
    To prove the converse inclusion we notice that for all $m\in M,$ $\rho(m)=\wot\lim_j \rho_c(\gamma_{f_{i_j}}(m))\in \overline{\rho_c(M)}^{\wot} = \rho_c(M)''.$
    Then $\rho(M)\sbe \rho_c(M)''$ and it follows that $\rho_c(M)'=\rho_c(M)'''\sbe \rho(M)'.$
\end{proof}

\section{Deferred Proofs}\label{sec: deferred proofs}

In this appendix, we provide the full details for the proofs of Lemmas \ref{lemma: the key lemma}, \ref{lem:equivalence preserved under diagonal tensor}, and \ref{lemma: pmap and the strong equivalente}.

\subsection{Proof of Lemma \ref{lemma: the key lemma}}
    Notice that $Z(\pi(A)'')=Z(\pi(A)')=Z(M)$.
    So, as discussed previously, it suffices to show that there exists an equivariant conditional expectation $P\colon  (\li{G} \bo Z(M))^c\to Z(M)^c$.
    
    Let $(Y,\rho,V)$ be a faithful and nondegenerate covariant representation of $\alpha\otmax^d\beta$ and split it into covariant representations $(Y,\rho_A,V)$ and $(Y,\rho_c,V)$ of $\alpha$ and $\beta$, that is,  $\rho_A(a)\rho_c(m)=\rho(a\otimes m)$ for $a\in A$ and $m\in M$. Since $\alpha\otmax^d \beta$ has the \wcp, $1\rho\rtimes \lambda V$ is faithful and 
    \[C:= 1\rho\rtimes \lambda V ((A\otmax M^c)\rtimes G)\sbe \bB(\ltwo{G}\otimes Y)\] is isomorphic to the full crossed product.
    Let $\iota^c\colon M^c\to \bB(X)$ be the inclusion and $\psi_0\colon C\to \bB(X)$ the unique *-homomorphism such that $\psi_0\circ (1\rho\rtimes \lambda V)=(\pi\times\iota^c)\rtimes U$; where $\pi\times \iota^c (a\otimes m)=\pi(a)m$.
    
    Fix a ccp extension $\psi\colon \bB(\ltwo{G}\otimes Y)\to \bB(X)$ of $\psi_0$ and name $D$ the multiplicative domain of $\psi$ \cite[Proposition 1.5.7 \& Definition 1.5.8]{BrownOzawa}.
    Recall that 
    \begin{align*}
        D & = \{ T\in \bB(\ltwo{G}\otimes Y)\colon \psi(T^*T)=\psi(T)^*\psi(T),\ \psi(TT^*)=\psi(T)\psi(T)^* \}\\
        & = \{ T\in \bB(\ltwo{G}\otimes Y)\colon \forall \ S\in \bB(\ltwo{G}\otimes Y),\ \psi(ST)=\psi(S)\psi(T),\ \psi(TS)=\psi(T)\psi(S) \}.
    \end{align*}
    From the first identity it follows that $C\sbe D$. Now fix $t\in G$.
    To prove that $\psi(\lambda V_t)=U_t$ we take $f\in C_c(G,A\otmax M^c)$ and $x\in X$.
    We have
    \begin{align*}
        \psi(\lambda V_t) (\pi\times\iota^c)\rtimes U(f) x
       & = \psi(\lambda V_t) \psi(1\rho\rtimes \lambda V(f)) x
       = \psi\left((\lambda V_t)1\rho\rtimes \lambda V(f)\right)x
       = \psi\left(1\rho\rtimes \lambda V(\tilde{\alpha}_t(f))\right)x\\
       &= (\pi\times\iota^c)\rtimes U(\tilde{\alpha}_t(f))x
       = U_t(\pi\times\iota^c)\rtimes U(f)x.
    \end{align*}
Since $(\pi\times\iota^c)\rtimes U$ is nondegenerate, the equalities above imply that $\psi(\lambda V_t)=U_t$.
In particular, $\psi$ is unital and this yields $\lambda V_t\in D$ (because $\lambda V_t$ and $U_t$ are unitaries).
Notice that this implies that $\psi$ is equivariant with respect to $\Ad(\lambda V)$ and $\Ad(U)$.

We claim that $\psi\circ (1\rho) = \pi\times\iota^c$.
Indeed, this is a straightforward consequence of the fact that for all $a\in A\otmax M^c$, $f\in C_c(G,A\otmax M^c)$ and $x\in X$,
 \begin{align*}
        \psi(1\rho(a)) (\pi\times\iota^c)\rtimes U(f) x
       & = \psi(1\rho(a)) \psi(1\rho\rtimes \lambda V(f)) x
       = \psi\left((1\rho(a))1\rho\rtimes \lambda V(f)\right)x\\
       & = \psi\left(1\rho\rtimes \lambda V(af)\right)x
       = (\pi\times\iota^c)\rtimes U(af)x
       = (\pi\times\iota^c(a)) (\pi\times\iota^c)\rtimes U(f)x.
    \end{align*}
The restriction of $\psi$ to the image of $1\rho$ is a representation, so $1\rho(A\otmax M^c)\sbe D$.
Moreover, for all $a\in A$, $\psi( 1\otimes \rho_A(a))=\psi(1\rho(a\otimes 1))=\pi(a)$.
Thus, for all $m\in M^c$, $\psi(1\otimes\rho_c(m))\pi(a) = \psi(1\otimes \rho(a\otimes m))=\pi(a)\iota^c(m)=\iota^m(c)\pi(a)$.
Replacing $a$ with an approximate unit of $A$ and taking limit in the weak operator topology (\wot) it follows that $\psi(1\otimes\rho_c(m)) = \iota^c(m)$.
This implies that the image of $1\otimes\rho_c$ is contained in $D$.

Let  $\overline{\rho_c}\colon M\to \bB(Y)$ be the extension of $\rho_c$ given by Lemma~\ref{lemma: extension of pic}.
Then $\rho_A(A)\sbe \rho_c(M^c)'=\overline{\rho_c}(M)'$ and $\overline{\rho_c}(M)\sbe \overline{\rho_c}(M)''\sbe \rho_A(A)'$.
Since $\rho_c$ is a *-homomorphism and $\overline{\rho_c}$ an extension of $\rho_c$, $M^c$ is contained in the multiplicative domain of $\rho$ and it follows that for all $T\in Z(M)$ and $m\in M^c$,  $\overline{\rho_c}(T)\rho_c(m)=\overline{\rho_c}(Tm)=\overline{\rho_c}(mT)=\rho_c(m)\overline{\rho_c}(T)$.
Hence, $\overline{\rho_c}(Z(M))\sbe \rho_A(A)'\cap \rho_c(M)'.$

Viewing $\overline{\rho_c}|_{Z(M)}$ as a ucp equivariant function from $Z(M)$ to $N:=\rho_A(A)'\cap \rho_c(M)'$, we can use Theorem~\ref{thm: equivariant vN tensor of equivariant maps} to get a ucp  and $\Ad(\lambda)\bo (\gamma|_{Z(M)})-\Ad(\lambda)\bo (\Ad(V)|_N)$ equivariant map  
\[\mu\colon \li{G} \bo Z(M)\to \li{G}\bo N \sbe  \bB(\ltwo{G}\otimes Y) \]
such that, for all $f\in \li{G}$ and $m\in Z(M)$, $\mu(f\otimes m)=f\otimes \overline{\rho_c}(m)$.

We claim that $\psi( \mu( \li{G} \bo Z(M) ) )\sbe Z(M)$.
To prove this we use the fact that
\[Z(M) = M\cap M'\equiv \pi(A)'\cap M' = (\pi\times\iota^c(A\otmax M^c))'.\]
Take $T\in \li{G} \bo Z(M)$ and $a\in A\otmax M^c$.
Since $1\rho(a)\in D$; $\psi(1\rho(a))=\pi\times\iota^c(a)$ and $\li{G}\bo N\sbe [1\rho(A\otmax M^c)]'$ we get $\psi(\mu(T))\pi\times\iota^c(a) = \psi(\mu(T)1\rho(a))= \psi(1\rho(a)\mu(T)) = \pi\times \iota^c(a)\psi(\mu(T)),$ so $\psi(\mu(T))\in Z(M).$

Since both $\mu$ and $\psi$ are equivariant and norm continuous, $\psi(\mu( (\li{G} \bo Z(M))^c ))\sbe Z(M)^c$ and we get a ccp equivariant map
\begin{align*}
  P & \colon (\li{G} \bo Z(M))^c\to Z(M)^c & T&\mapsto \psi(\mu(T))
\end{align*}
For $m\in Z(M)^c\sbe \li{G} \bo Z(M)$ we have $m\in M^c$ and $P( 1\otimes m ) = \psi(1\otimes \rho_c(m)) = \iota^c(m) = m.$
This shows that $P$ is an equivariant conditional expectation.

\subsection{Proof of Lemma \ref{lem:equivalence preserved under diagonal tensor}}
Let $\cX=\{X_t\}_{t\in G}$ be a $\cB$--$\cC$ equivalence bundle, and let $\bL(\cX)=\{\bL(X_t)\}_{t\in G}$ denote its linking bundle as in \cite[Section~3.1]{AbFrrEquivalence}. 
Since $B_e$ and $C_e$ are hereditary \Cstar-subalgebras of $\bL(X_e)$, we have
\[
   A_e\otmax B_e \;\subseteq\; A_e\otmax \bL(X_e) \;\supseteq\; A_e\otmax C_e,
\]
so that $\cA\otmax^d \cB$ and $\cA\otmax^d \cC$ can be naturally regarded as Fell subbundles of $\cA\otmax^d \bL(\cX)$.

Moreover, $\cX$ sits as a Banach subbundle of $\bL(\cX)$, and hence
\[
  \cA\otmax^d \cX
  := \Big\{ \cspn\{ a\otimes x : a\in A_t,\; x\in X_t \} \Big\}_{t\in G}
  = \{A_t\otmax X_t\}_{t\in G}
\]
is a Banach subbundle of $\cA\otmax^d \bL(\cX)$. 
By restricting the multiplication and involution from $\cA\otmax^d \bL(\cX)$, this subbundle acquires a natural structure of an $\cA\otmax^d\cB$--$\cA\otmax^d\cC$ equivalence bundle.

To check fullness, set
\[
   S := \cspn \Big\{ {}_{\cA\otmax^d \cB}\langle a\otimes x,\; b\otimes y\rangle 
   : a\in A_t,\, b\in A_\tmu,\, x\in X_t,\, y\in X_\tmu,\, t\in G \Big\}.
\]
A direct computation shows that
\[
   S = \cspn \{ ab^* \otimes {}_{\cB}\langle x,y\rangle : a\in A_t,\; b\in A_\tmu,\; x\in X_t,\; y\in X_\tmu,\; t\in G \}.
\]

Now suppose $\cA$ is the semidirect product bundle associated with a \Cstar-action $\alpha$ of $G$ on a \Cstar-algebra $A$. 
Given $z\delta_t \in A\delta_t \cong A_t$, one can write $z=z_1z_2$ with $z_1,z_2\in A$, and then
\[
   z\delta_e = (z_1\delta_t)\,(\alpha_{t^{-2}}(z_2)^*\delta_\tmu)^*.
\]
This implies that $S=A_e\otmax B_e$. 
A similar computation shows that
\[
   \cspn \Big\{ \langle a\otimes x, b\otimes y\rangle_{\cA\otmax^d \cC} : a\in A_t,\, b\in A_\tmu,\, x\in X_t,\, y\in X_\tmu,\, t\in G \Big\}
   = A_e\otmax C_e.
\]

Thus conditions (7R) and (7L) of \cite[Definitions~2.2]{AbFrrEquivalence} are satisfied. 
All the remaining conditions follow automatically from the fact that the operations on $\cA\otmax^d \cX$ were inherited from the Fell bundle $\cA\otmax^d \bL(\cX)$. 
Hence $\cA\otmax^d\cB$ and $\cA\otmax^d\cC$ are equivalent.

\subsection{Proof of Lemma \ref{lemma: pmap and the strong equivalente}}
    Let $\cA=\{A_t\}_{t\in G}$ and $\cB=\{B_t\}_{t\in G}$ be Fell bundles and let $\cX=\{X_t\}_{t\in G}$ be a strong $\cA$--$\cB$ equivalence bundle with inner products ${}_\cA\langle \cdot,\cdot\rangle$ and $\langle \cdot,\cdot\rangle_\cB$. Since strong equivalence is an equivalence relation \cite[Theorem~A.1]{AbBssFrrMorita}, we may assume, without loss of generality, that $\cA$ has the positive $1$\nobreakdash-approximation property. Let $\{\xi_i\}_{i\in I}\subseteq C_c(G,A_e)$ be a net as in the definition of this property.

    Let $\Gamma$ denote the set of cross-sections of $\cB$ of the form $r\mapsto \langle x,f(r)y\rangle_\cB$, with $x,y\in X_e$ and $f\in C_c(\cA)$. Clearly $C_c(G)\Gamma\subseteq \Gamma$, hence $\Gamma':=\spn\Gamma$ satisfies $C_c(G)\Gamma'\subseteq \Gamma'$. For all $r\in G$, using \cite[Notation~2.1]{AbBssFrrMorita} and the definition of strong equivalence, one has
    \[
       B_r = B_rB_r^*B_r = B_r\langle X_r,X_r\rangle_\cB = \langle X_rB_{r^{-1}},X_r\rangle_\cB \subseteq \langle X_e, {}_\cA\langle X_r,X_r\rangle X_r\rangle_\cB \subseteq \langle X_e,A_rX_e\rangle_\cB,
    \]
    so that $\{u(r):u\in\Gamma'\}=\spn\{\langle x,ay\rangle_\cB: x,y\in X_e,\,a\in A_r\}$ is dense in $B_r$.

    By combining the Axiom of Choice with Lemma~\ref{lemma: convergence on gamma suffices}, we get that it suffices to show that for every finite $U\subseteq\Gamma$ and $\varepsilon>0$, the set
    \[
       S(U,\varepsilon):=\{\eta\in C_c(G,B_e): \|\eta\|_2\leq 1,\ \|\Phi_\eta(f)-f\|_\infty\leq \varepsilon\ \forall f\in U\}
    \]
    is nonempty.

    Fix $U\subseteq \Gamma$ finite and $\varepsilon>0$. Since $X_e$ is an $A_e$--$B_e$ equivalence module and every approximate unit of $B_e$ is a strong approximate unit of $\cB$, \cite[Lemma~2.8]{AbFrrEquivalence} yields the existence of $x_1,\dots,x_n\in X_e$ such that, for $b:=\sum_{k=1}^n\langle x_k,x_k\rangle_\cB$, we have $0<\|b\|<1$ and
    \[
       \|bu(r)b-u(r)\|<\varepsilon/2 \quad \forall (r,u)\in G\times U.
    \]
    By assumption, for sufficiently large $i\in I$ we have
    \[
       \|\Phi_{\xi_i}({}_\cA\langle x_ku(r),x_l\rangle)-{}_\cA\langle x_ku(r),x_l\rangle\|<\tfrac{\varepsilon}{2n^2}
    \]
    for all $k,l=1,\dots,n$, all $r\in G$, and all $u\in U$.

    Define $\eta\in C_c(G,B_e)$ by $\eta(s)=\sum_{k=1}^n\langle x_k,\xi_i(s)x_k\rangle_\cB$. Then
    \[
       \langle\eta,\eta\rangle=\int_G \sum_{k,l=1}^n \langle \xi_i(s)x_k, {}_\cA\langle x_k,x_l\rangle \xi_i(s)x_l\rangle_\cB\,ds.
    \]
    Viewing $X_e^n$ as a $\bM_n(A_e)$--$B_e$ Hilbert bimodule, for the element $x=(x_1,\ldots,x_n)\in X_e^n$ we have  ${}_{\bM_n(A_e)}\langle x,x\rangle = ({}_\cA\langle x_i,x_j\rangle)_{i,j=1}^n\geq 0$; $\|({}_\cA\langle x_i,x_j\rangle)_{i,j=1}^n\|=\|x\|^2 = \|\sum_{i=1}^n \langle x_i,x_i\rangle_\cB\|\leq 1$ and it follows that 
    \[
      0\leq \sum_{k,l}\langle \xi_i(s)x_k, {}_\cA\langle x_k,x_l\rangle \xi_i(s)x_l\rangle_\cB \leq \|x\| \sum_{k=1}^n \langle x_k,\xi_i(s)^*\xi_i(s)x_k\rangle_\cB.
    \]
    Integrating with respect to $s$ one gets
    \[
       \langle\eta,\eta\rangle \leq \sum_{k=1}^n \langle x_k, \int_G \xi_i(s)^*\xi_i(s)\, ds x_k\rangle_\cB \leq \sum_{k=1}^n \|\xi_i\|_2^2 \langle x_k, x_k\rangle_\cB =  \|\xi_i\|_2^2\,b,
    \]
    and therefore $\|\eta\|_2\leq \|\xi_i\|_2\leq 1$.

    Now let $(r,u)\in G\times U$, say $u(r)=\langle x_u,f_u(r)y_u\rangle_\cB$. A direct computation gives
    \[
       \Phi_\eta(u(r))=\sum_{k,l=1}^n\langle x_k,\Phi_{\xi_i}({}_\cA\langle x_ku(r),x_l\rangle)x_l\rangle_\cB.
    \]
    Combining this with $\|x_k\|^2=\|\langle x_k,x_k\rangle_\cB\|\leq \|b\|<1$, we obtain
    \begin{align*}
       \|\Phi_\eta(u(r))-u(r)\| &\leq \|bu(r)b-u(r)\| + \Big\|\sum_{k,l}\langle x_k,(\Phi_{\xi_i}-\id)({}_\cA\langle x_ku(r),x_l\rangle)x_l\rangle_\cB\Big\| \\
       &< \tfrac{\varepsilon}{2} + \sum_{k,l=1}^n \|x_k\|\|x_l\|\,\tfrac{\varepsilon}{2n^2} \\
       &< \tfrac{\varepsilon}{2} + \tfrac{\varepsilon}{2}=\varepsilon.
    \end{align*}
    Thus $S(U,\varepsilon)\neq\emptyset$, and the lemma follows.


\begin{thebibliography}{10}

\bibitem{abadie1997tensorv3}
Fernando Abadie.
\newblock Tensor products of fell bundles over discrete groups.
\newblock {\em arXiv preprint funct-an/9712006v3}, 1997.

\bibitem{AbPhd}
Fernando Abadie.
\newblock {\em Sobre a\c c\~oes parciais, fibrados de {F}ell, e grup\'oides}.
\newblock PhD thesis, Uni\-ver\-sidade de {S}\~ao {P}aulo, September 1999.

\bibitem{Ab03}
Fernando Abadie.
\newblock Enveloping actions and {T}akai duality for partial actions.
\newblock {\em J. Funct. Anal.}, 197(1):14--67, 2003.

\bibitem{abadie2024conditional}
Fernando Abadie.
\newblock Conditional expectations on {F}ell bundles.
\newblock {\em arXiv preprint arXiv:2412.12446}, 2024.

\bibitem{AbBssFrrMorita}
Fernando Abadie, Alcides Buss, and Dami\'{a}n Ferraro.
\newblock Morita enveloping {F}ell bundles.
\newblock {\em Bull. Braz. Math. Soc. (N.S.)}, 50(1):3--35, 2019.

\bibitem{AbFrrEquivalence}
Fernando Abadie and Dami\'{a}n Ferraro.
\newblock Equivalence of {F}ell bundles over groups.
\newblock {\em J. Operator Theory}, 81(2):273--319, 2019.

\bibitem{ADaction1979}
C.~Anantharaman-Delaroche.
\newblock Action moyennable d'un groupe localement compact sur une alg\`ebre de
  von {N}eumann.
\newblock {\em Math. Scand.}, 45(2):289--304, 1979.

\bibitem{ADactionII1982}
C.~Anantharaman-Delaroche.
\newblock Action moyennable d'un groupe localement compact sur une alg\`ebre de
  von {N}eumann. {II}.
\newblock {\em Math. Scand.}, 50(2):251--268, 1982.

\bibitem{ADsystemes1987}
Claire Anantharaman-Delaroche.
\newblock Systemes dynamiques non commutatifs et moyennabilit{\'e}.
\newblock {\em Mathematische Annalen}, 279(2):297--315, 1987.

\bibitem{bedos2024positiveMZ}
Erik B{\'e}dos and Roberto Conti.
\newblock Positive definiteness and fell bundles over discrete groups.
\newblock {\em Mathematische Zeitschrift}, 311, 2025.

\bibitem{BrownOzawa}
Nathanial~P. Brown and Narutaka Ozawa.
\newblock {\em {$C^\ast$}-algebras and finite-dimensional approximations},
  volume~88 of {\em Graduate Studies in Mathematics}.
\newblock American Mathematical Society, Providence, RI, 2008.

\bibitem{BssEff_amenability}
Alcides Buss, Siegfried Echterhoff, and Rufus Willett.
\newblock Amenability and weak containment for actions of locally compact
  groups on {$C^*$}-algebras.
\newblock {\em Mem. Amer. Math. Soc.}, 301(1513):v+88, 2024.

\bibitem{BssFrrSeh}
Alcides Buss, Dami\'{a}n Ferraro, and Camila~F. Sehnem.
\newblock Nuclearity for partial crossed products by exact discrete groups.
\newblock {\em J. Operator Theory}, 88(1):83--115, 2022.

\bibitem{Ex97twisted}
Ruy Exel.
\newblock Twisted partial actions: a classification of regular
  {$C^\ast$}-algebraic bundles.
\newblock {\em Proc. London Math. Soc. (3)}, 74(2):417--443, 1997.

\bibitem{ExlibroAMS}
Ruy Exel.
\newblock {\em Partial dynamical systems, Fell bundles and applications},
  volume 224 of {\em Mathematical Surveys and Monographs}.
\newblock American Mathematical Soc., Providence, RI, 2017.

\bibitem{ExNg}
Ruy Exel and Chi-Keung Ng.
\newblock Approximation property of {$C^\ast$}-algebraic bundles.
\newblock {\em Mathematical Proceedings of the Cambridge Philosophical
  Society}, 132:509--522, 5 2002.

\bibitem{FlDr88}
J.~M.~G. Fell and R.~S. Doran.
\newblock {\em Representations of {$^*$}-algebras, locally compact groups, and
  {B}anach {$*$}-algebraic bundles.}, volume 125--126 of {\em Pure and Applied
  Mathematics}.
\newblock Academic Press Inc., Boston, MA, 1988.

\bibitem{Ferraro_2024}
Dami\'{a}n Ferraro.
\newblock Cross-sectional c*-algebras associated with subgroups.
\newblock {\em Canadian Journal of Mathematics}, page 1–26, 2024.

\bibitem{Ikunishi88}
Akio Ikunishi.
\newblock The {$W^*$}-dynamical system associated with a {$C^*$}-dynamical
  system, and unbounded derivations.
\newblock {\em J. Funct. Anal.}, 79(1):1--8, 1988.

\bibitem{LauPat1991}
Anthony~T. Lau and Alan L.~T. Paterson.
\newblock Amenability for twisted covariance algebras and group
  {$C^*$}-algebras.
\newblock {\em J. Funct. Anal.}, 100(1):59--86, 1991.

\bibitem{mckee2020amenable}
Andrew McKee and Reyhaneh Pourshahami.
\newblock Amenable and inner amenable actions and approximation properties for
  crossed products by locally compact groups.
\newblock {\em Canadian Mathematical Bulletin}, pages 1--19, 2020.

\bibitem{NagTomCP}
Masaru Nagisa and Jun Tomiyama.
\newblock Completely positive maps in the tensor products of von {N}eumann
  algebras.
\newblock {\em J. Math. Soc. Japan}, 33(3):539--550, 1981.

\bibitem{ozawa2021characterizations}
Narutaka Ozawa and Yuhei Suzuki.
\newblock On characterizations of amenable {$\rm C^*$}-dynamical systems and
  new examples.
\newblock {\em Selecta Math. (N.S.)}, 27(5):Paper No. 92, 29, 2021.

\bibitem{takesaki1979theory}
Masamichi Takesaki.
\newblock {\em Theory of operator algebras {I}}.
\newblock Springer, New York, 1979.

\bibitem{Tomiyama_tensor_vN_algebras}
Jun Tomiyama.
\newblock On the tensor products of von {N}eumann algebras.
\newblock {\em Pacific J. Math.}, 30:263--270, 1969.

\end{thebibliography}
\bibliographystyle{plain}

\end{document}